\documentclass[10pt]{article}

\usepackage{amsthm}
\usepackage{amssymb}
\usepackage{amsmath}
\usepackage{graphicx}
\usepackage{eufrak} % Frakturgraphie
\usepackage{comment}
\usepackage{mathrsfs}
\usepackage{eucal}  % Kalligraphie
\usepackage{mathrsfs}
\usepackage{mathtools}
\usepackage{ifthen}
\newboolean{skript}
\newcommand{\beweis}{\ifthenelse{\boolean{skript}} {
\begin{proof}[Beweis]
\mbox{}\vfill \mbox{}
\end{proof}
\pagebreak }{} }

\allowdisplaybreaks[3]

\providecommand{\phinneuu}{\frac{1}{\check{{\phi}}_n(u) } }
\providecommand{\phinu}{\frac{1}{\phi(u)}}

\providecommand{\xfmk}{x_{f_{m,k}}}

\providecommand{\thetaneu}{\widehat{{\theta}}}

\providecommand{\phinneuu}{\frac{1}{{\check{\phi}}_n(u) } }
\providecommand{\Fourier}{\mathcal{F}}  

\providecommand{\phinu}{\frac{1}{\phi(u)} }

\usepackage{multirow}
\usepackage{multicol}

\providecommand{\kerndiffmk}{\Fourier \kf \left( u/k\right) - \Fourier \kf \left( u/m \right) }
\providecommand{\Kerndifmk}{ \Fourier \kf \left( u/k \right) - \Fourier \kf \left( u/m \right) }

\renewcommand{\hat}{\widehat}
\renewcommand{\tilde}{\widetilde}

\usepackage{titlesec}
\usepackage{textfit}

\providecommand{\neu}{\hspace*{-0,08cm}\scaletoheight{0,07cm}{N}}
\providecommand{\neukap}{\hspace*{-0,08cm}\scaletoheight{0,07cm}{NK}}
\providecommand{\neurei}{\hspace*{-0,08cm}\scaletoheight{0,07cm}{NR}}
\providecommand{\kap}{\hspace*{-0,08cm}\scaletoheight{0,07cm}{K}}

\providecommand{\Hoe}{H\"older}

\providecommand{\Levy}{L\'evy}

\usepackage{geometry}
\usepackage[round]{natbib}
 \numberwithin{equation}{section}

\newtheorem{satz}{Satz}[section]

\newtheorem{definition}[satz]{Definition}

\newtheorem{theorem}[satz]{Theorem}

\newtheorem{proposition}[satz]{Proposition}

\newtheorem{corollary}[satz]{Corollary}
\newtheorem{lemma}[satz]{Lemma}

\theoremstyle{definition}

\DeclareMathOperator{\E}{{\mathbb E}}

\DeclareMathOperator{\R}{{\mathbb R}}

\DeclareMathOperator{\Z}{{\mathbb Z}}
\DeclareMathOperator{\N}{{\mathbb N}}

\DeclareMathOperator{\PP}{{\mathbb P}}

\DeclareMathOperator{\Q}{{\mathbb Q}}

\DeclareMathOperator{\supp}{supp}

\DeclareMathOperator{\arginf}{arginf}

\DeclareMathOperator{\argmin}{argmin}
\DeclareMathOperator{\Var}{Var} \DeclareMathOperator{\Cov}{Cov}

\DeclareMathOperator{\kf}{K}
\DeclareMathOperator{\lk}{L}
\DeclareMathOperator{\pen}{pen}

\DeclareMathOperator{\strafH}{H}

\renewcommand{\emph}{\textit}
\renewcommand{\d}{\ensuremath{\,\textnormal{d}}}
\providecommand{\eps}{\varepsilon}
\renewcommand{\phi}{\varphi}

\renewcommand{\supset}{\supseteq}

\providecommand{\Fourier}{\mathcal{F}}

\providecommand{\emphit}{\textnormal}

\begin{document}

\title{Adaptive nonparametric estimation for L\'evy processes observed at low frequency}

\author{Johanna Kappus\\[0,2cm]
\small \textit{
Institut f\"ur Mathematik,
Universit\"at Rostock,
18051 Rostock,Germany}\\
\small \textit{Email: johanna.kappus@uni-rostock.de, Phone: 49 381 498-6604}
}
\date{\ }
\maketitle
\vspace*{-0,8cm}
\noindent\rule{\textwidth}{1pt}
\begin{abstract}
\noindent This article deals with adaptive nonparametric estimation  for L\'evy processes observed at low frequency.  
For general linear functionals of the L\'evy measure, we construct kernel estimators, provide upper risk bounds and derive rates of convergence under regularity assumptions.

\noindent Our focus lies on the adaptive choice of the bandwidth, using model selection techniques. We face here a non-standard problem of model selection with unknown variance. A new approach towards this problem is proposed, which also allows a straightforward generalization to a classical density deconvolution framework. 
\end{abstract}
\noindent\rule{\textwidth}{1pt}
\section{Introduction}
\label{lab}
\noindent L\'evy processes, continuous time stochastic processes with independent and stationary increments, are the building blocks for a large number of continuous time models with jumps which play an important role, for example, in the modeling of financial data. See \citet{Cont-Tankov} for an overview of the topic. The problem of estimating the characteristics of a L\'evy process is thus  not only a topic of great theoretical relevance, but also an important issue for practitioners and has received considerable attention over the 
past decade.
Starting from the work by 
\citet{Belomestny_Reiss},  nonparametric estimation methods for L\'evy processes have been considered in a number of articles in the past few years. Let us mention \citet*{Neumann_Reiss},  \citet{Gugushvili_1,Gugushvili_2},  \citet{Comte_Genon_1, Comte_Genon_2, Comte_Genon_3}, \citet*{Figueroa_1} and \citet*{Figueroa_2},  \citet*{Nickl_Reiss}, and \citet*{Belomestny_2}.
For results on time changed L\'evy processes, see \citet{Belomestny_tcLevy}. 

In the present work, we focus on the adaptive  estimation of the jump  measure for L\'evy processes observed at low frequency. 
The following statistical model is being considered: A L\'evy process $X$  having finite variation  on compact sets and  finite first moments is observed at discrete, equidistant time points. We investigate the nonparametric  estimation of linear functionals of the finite signed measure $\mu(\d x) = x\nu(\d x)$, with $\nu$ denoting the L\'evy measure.
Kernel estimators  are constructed and upper bounds on the corresponding risk are derived. Our main concern is  to provide a strategy for the data driven choice of the smoothing parameter, using techniques of model selection via penalization. 

The model selection approach to adaptive estimation  has been extensively studied in the literature, starting from the work by  Birg\'e and Massart in the late 90's, see, for example, \citet*{Massart_Birge} and \citet*{Birge}.  The model selection point of view essentially differs from other existing methods, typically in the spirit of Lepski, see  e.g. \cite{Lepski_1,Lepski_2},  in the sense that the problem is considered from a non-asymptotic perspective. We refer to \citet*{Massart_Birge} for a detailed discussion and systematic comparison. 

Recently, there is a strong tendency to apply model selection techniques in adaptive estimation problems for L\'evy processes, see \citet{Figueroa_1} and subsequent papers,  and   \citet{Comte_Genon_2,Comte_Genon_1,Comte_Genon_3}. However, the above mentioned papers have mainly focused on a situation where continuous time or high frequency observations of the process are available. In the present work, we consider and fully solve the problem of adaptation in a low frequency framework.

It is well known that, depending on the nature of the 
observations, there are two fundamentally different approaches to the estimation of the  L\'evy measure: For continuous time or high frequency data, the jumps are directly observable or observable in the limit, so it is possible to use the empirical jump measure as an estimator of 
the true underlying jump measure. The estimation procedure can then be understood in analogy with a density estimation problem. 
When low frequency observations of the process are available, one has to exploit the structural properties of infinitely divisible laws and faces a more complicated deconvolution type structure. 

In a density deconvolution framework, adaptive estimation by model selection has been considered 
by  \citet*{Comte_Rosenholz} and by \citet*{Comte_Lacour}. For estimating linear functionals in a density deconvolution model, see  \citet*{Butucea}. However, in those  papers, the distribution of the noise is assumed to be known and the adaptive procedure crucially depends on the fact that the variance term is feasible, which is no longer true in the present model.
 Indeed, the situation which we consider in the present work can be understood in analogy with a deconvolution problem with unknown distribution of the noise. 

Rates of convergence for deconvolution problems with unknown error density have been studied, since the late 90's, for example, in 
\citet{Neumann, Neumann_2}, \citet{Meister_2, Alex_Meister_habil} or \citet{Johannes}. However,  few literature is available on adaptive estimation in deconvolution problems with unknown distribution of the errors. We are aware of the work by  \citet*{Schwarz_Johannes} in a circular deconvolution model and of  \citet*{Comte_Lacour}. We propose here a new approach to dealing with the unknown variance. Although our results are formulated for estimation problems in a L\'evy model, they can be generalized to a classical density deconvolution framework. 

 Compared to the reasoning in \citet{Comte_Lacour}, we can avoid the loss of a polynomial factor. Moreover, unlike in \citet*{Comte_Lacour} and \citet*{Schwarz_Johannes}, our reasoning does not rely on certain semiparametric assumptions on the decay of the characteristic function.  This  allows a fully general treatment. 

The main technical step involved in our arguments relies on an application of the Talagrand inequality, which allows to give control on the fluctuation of the empirical characteristic function in the denominator around its target uniformly on the real line, thus improving a classical pointwise result presented in  \citet*{Neumann}.

This paper is organized as follows:  In Section \ref{Abschnitt_Modell},  we start by specifying the statistical model and technical assumptions.  In Section  \ref{Abschnitt_Schaetzer_und_Raten} kernel  estimators  are introduced and upper bounds on  the corresponding risk are provided. From this, rates of convergence are derived under regularity assumptions.  The adaptive estimation procedure and hence the main result of this paper is then presented in Section \ref{Abschnitt_adaptive_Schaetzung}.
All proofs are postponed to Section \ref{Beweisteil}.
%%%%%%%%%%%%%%%%%%%%%%%

%%%% Abschnit statistisches Modell

%%%%%%%%%%%%%%%%%%%%%%%%%%%%
\section{Statistical model}
\label{Abschnitt_Modell}
A L\'evy process $X=\{X_t: t\in \R^+\}$ taking values in $\R$ is observed at equidistant time points \mbox{$\Delta, \cdots, 2n\Delta$}, with $\Delta>0$ fixed. 

We  work under the following technical  assumptions:
\textit{
\begin{itemize}
 \item[(A1)]  $X$ is of pure jump type.
\item [(A2)]  $X$ has moderate activity of small jumps in the sense that the following holds
true for the  \mbox{L\'evy measure $\nu$}:
\begin{equation}
\label{Bedingung_LM_1}
 \int_{\{|x|\leq 1\}} \limits |x| \nu(\d x) < \infty.
\end{equation}
\item[(A3)] $X$ has no drift component.
\item[(A4)] For one and hence for any $t> 0$, $X_t$ has a finite second moment. This is equivalent to stating that
\begin{equation}
\label{Bedingung_LM_2}
 \int |x|^2 \nu(\d x) <\infty.
\end{equation}
\end{itemize}}
Imposing the assumptions (A1) and (A2) is equivalent to stating that $X$  has finite variation on compact sets. 

Under (A1)-(A4),  the characteristic 
function of $X_{\Delta}$ is given by
\begin{equation*}
\label{CF}
 \phi_{\Delta}(u):= \E \left[ e^{iuX_{\Delta}} \right] = e^{\Delta \Psi(u)},
\end{equation*}
with characteristic exponent
\begin{equation}
 \label{Spezialform_CE}
\Psi(u) = \int \left(e^{iux}-1  \right) \nu(\d x) = \int \frac{e^{iux}-1}{x} x\nu(\d x).
\end{equation}
(A proof can be found, for example, in Chapter 2 in  \citet*{Sato}).
The process is thus fully described by the signed measure $\mu(\d x):= x \nu(\d x)$, which is finite thanks to
(\ref{Bedingung_LM_1}) and (\ref{Bedingung_LM_2}).

It is worth mentioning that the conditions (A1)-(A3) can be relaxed, but at the cost of substantially  complicating the structure of the estimator and proofs without allowing much further insight into the nature of the problem.  See  \citet*{Neumann_Reiss} or  \citet{Belomestny_tcLevy} for estimation strategies in a more general framework.

Our goal is to estimate a linear functional of $\mu$.  Let $f$ be a distribution and assume that one of the following conditions is met:
\textit{
\begin{itemize}
 \item[(F1)] $f$ is regular  and can hence be identified with a function.  $f\in \lk^1(\R)$ and $\sup_{x\in \R} |f(x)|<\infty$.
\item[(F2)] $f$ is compactly supported, with order $k$,  and for some open interval  $D=(d_1, d_2)$ with  $\supp(f)\subseteq D$, the restriction $\mu|_D$ possesses a density  \mbox{$g_D\in C^k(D)$}.
\end{itemize}  }
Then the parameter of interest is
\[
\theta:=\langle f, \mu \rangle.
\]
For the definition of the order of a distribution and a concise overview of the theory of distributions, we refer to   \citet{Rudin}. For a detailed introduction, see \citet{Jantscher}. 

To better understand the meaning of the abstract assumptions (F1) and (F2), let us mention that under (F1), we simply have 
\[
\theta = \int f(x) \mu(\d x)
\]
and that (F2) covers typical problems such as point estimation or the estimation of derivatives.  For notational convenience, we will often write, formally,  $\int f(x)\mu(\d x):= \langle f, \mu\rangle$ even though $f$ may be nor-regular. 

The special case of testing $\mu$ with smooth functions has been considered in \citet*{Neumann_Reiss} and point estimation has been dealt with in \citet{Belomestny_tcLevy}. 

However, to the best of our knowledge, the estimation of arbitrary linear functionals of $\mu$ has not yet been treated in full generality. 

Moreover, the problem of adaptive estimation, which is the main concern of the present work, has not been treated in the above mentioned papers. 
%%%%%%%%%%%%%%%%%%%%%%%
%%%%%   Abschnitt Schätzer und Raten
%%%%%%%%%%%%%%%%%%%%%%%%%%%
\section{Estimation procedure, risk bounds and rates of convergence}
\label{Abschnitt_Schaetzer_und_Raten}
\subsection{Construction of the estimator  and non-asymptotic risk bounds}
Formula (\ref{Spezialform_CE}) allows to recover the Fourier transform $\Fourier \mu$ of $\mu$  by derivating  the characteristic exponent,
\begin{equation*}
  \Psi'(u)
= i \int e^{iux} \mu(\d x) =i \Fourier \mu(u).
\end{equation*}
In terms of the characteristic function and its derivative,
\begin{equation}
 \label{Fouriertransformierte_mu}
\Fourier \mu(u)  =
 \frac{\frac{1}{\Delta} \phi'_{\Delta}(u)}{i \phi_{\Delta}(u)}.
\end{equation}
Recall that the characteristic function of an infinitely divisible law possesses no zeros (see Lemma 7.5 in 
\citet{Sato}), so
dividing by $\phi_{\Delta}$ is not critical.

By the definition of a L\'evy process, the  increments $Z_{\Delta,k}:= X_{k\Delta}- X_{(k-1)\Delta},$  $k=1,\cdots, 2n$   form i.i.d. copies of $X_{\Delta}$.
We can thus define the empirical versions of $\phi_{\Delta}$ and $\phi'_{\Delta}$, setting
\begin{equation*}
\label{emp_char_fkt}
\hat{\phi}_{{\Delta,n} }(u):= \frac{1}{n } \sum_{k=1}^{n} e^{iu Z_{\Delta, k}}
\end{equation*}
and
\begin{equation*}
\label{emp_char_fkt_abl}
\hat{\phi}'_{{\Delta,n}}(u):= \frac{1}{n} \sum_{k=n+1}^{2 n} i Z_{\Delta, k} e^{iuZ_{\Delta, k }}.
\end{equation*}
Splitting the observations into two independent samples will be crucial for the adaptive procedure, see the proof of Proposition \ref{wichtigstes_Hilfsresultat} in Section \ref{Beweisteil_adaptiv}.

Following the approach in \citet{Comte_Genon_1} and \citet{Comte_Lacour}, we  replace the empirical characteristic function in the denominator by its truncated  version, setting
\begin{equation*}
\label{Schaetzer_Nenner}
\frac{1}{\tilde{\phi}_{{\Delta,n}}(u)}:= \frac{1(|\hat{\phi}_{{\Delta,n}}(u)|\geq (\Delta n)^{-1/2} )}
{\hat{\phi}_{{\Delta, n}}(u)}.
\end{equation*}
This definition has originally been introduced in  \citet{Neumann}.

Parseval's identity permits 
to express the quantity of interest in the Fourier domain, 
\begin{equation}
\label{Parseval}
\theta = \frac{1}{2\pi} \int \Fourier f(-u) \Fourier \mu(u) \d u.
\end{equation}
Recall, at this point, that the Fourier transform of a compactly supported distribution is a function, see Theorem 7.23 in \citet{Rudin}. 

Let $\kf$ be a kernel function. We use the notation $\kf_h(x) := h^{-1} \kf\left( x/h    \right)$.  
$\kf$ is chosen such that 
\textit{
\begin{itemize}
\item[(K1)] For any $h>0$,  $\Fourier \kf_h \Fourier f(-\cdot) \in \lk^1(\R)$. 
\item[(K2)] If $f$ is non-regular, with order $k$, $\kf$ is $k$-times continuously differentiable. 
\end{itemize}  }
Formula (\ref{Fouriertransformierte_mu}) and formula (\ref{Parseval}) suggest to define the kernel estimator 
 \begin{equation*}
 \label{Definition_allgemeiner_Kernschaetzer}
\hat{\theta}_{\Delta,h,n}:= \frac{1}{2\pi} \int \Fourier f (-u) \Fourier \kf_h (u) \frac{\frac{1}{\Delta}\hat{\phi}_{{\Delta,n}}'(u)}
{i \tilde{\phi}_{{\Delta,n}}(u)}\d u.
\end{equation*}
The following upper bound can be derived for the squared risk: 
\begin{theorem}
\label{Hauptsatz_allgemeiner_Kernschaetzer_Levy} Assume that (A1)-(A4),
(F1) or (F2)  and (K1)+(K2)  are met.  Then 
\begin{align*}
\label{obere Schranke}
\notag&\phantom{\leq}  \E \left[ \left|\theta- \hat{\theta}_{\Delta,h, n}\right|^2 \right] \\
\notag  &\leq 
2 \left| \int f(x) \mu(\d x) - \int f(x) \left( \kf_h\ast \mu\right)(\d x)\right|^2 \\
 &\phantom{\leq} +    \frac{T^{-1}}{2 \pi^2 }\left\{ C_1\int
 |\Fourier \kf (hu)|^2 \left|\frac{\Fourier f(-u)}{\phi_{\Delta}(u)}\right|^2 \d u
\wedge  C_2 \left(\int |\Fourier \kf ( hu)|\left|\frac{\Fourier f(-u)}{\phi_{\Delta}(u)}\right| \d u\right)^2
\right\}, 
\end{align*}
with $T:= \Delta n$ and   with constants
\begin{equation}
 \label{Definition_Konstante_1}
C_1=  C  \left(  \int |\Psi''(x)|\d x + 2 \int |\Psi'(x)|^2\d x \right)\leq \infty
 \end{equation}
 and
 \begin{equation}
 \label{Definition_Konstante_2}
 C_2= C \left(  \|\Psi''\|_{\infty} +  2 \|\Psi'\|_{\infty}^2\right) <\infty
 \end{equation}
for a  universal positive constant $C$. 
\end{theorem}
 For the special case of point estimation, Theorem  \ref{Hauptsatz_allgemeiner_Kernschaetzer_Levy} is in accordance with the results found in \citet{Belomestny_tcLevy}. It should also be compared to the deconvolution framework which is considered  in  \citet{Butucea}. 
%%%%%%%%%%%%%%%%%%%%%%%%%%%%%%
%%%% Unterabschnitt Raten
%%%%%%%%%%%%%%%%%%%%%%5
\subsection{Rates of convergence}
We investigate here the rates of convergence which can be derived from the upper risk bounds given in 
Theorem \ref{Hauptsatz_allgemeiner_Kernschaetzer_Levy}  under regularity assumptions on $\mu$ and $f$
and on the decay of the characteristic function. 

Recall the following definitions: For $a\in \R$ and $M> 0$, the  \textit{Sobolev class} $\mathcal{S}(a,M)$  consists of those   tempered distributions, for which 
\begin{equation*}
\int \left| \Fourier f(-u) \right|^2  \left(1+|u|^2\right)^a \d u \leq M. 
\end{equation*}
Let $\langle a \rangle:= \max\{k\in \N: k< a\}$. For an open subset 
$D\subseteq \R$, and positive constants $a,L$ and $R$,   the \textit{\Hoe \  class}
$\mathcal{H}_{D}(a, L, R)$ consists of those functions  $f$ for which \mbox{$\sup_{x\in D}|f(x)|\leq R$ holds and  $f\large|_D$} is $\langle a \rangle$-times continuously differentiable, with 
\begin{equation*}
\sup_{x,y\in D, x\not=y}  \frac{ |f^{(\langle a\rangle)}(x) - f^{(\langle a \rangle)}(y)|}{|x-y|^{a- \langle a \rangle} }\leq L .
\end{equation*}
A kernel $\kf$ is called a \emphit{$k$-th order kernel}, if for all integers $ 1\leq m < k$,
\begin{equation}
\label{Ordnung_Kern_1}
 \int x^m K(x) \d x=0
\end{equation}
and moreover,
\begin{equation*}
\label{Ordnung_Kern_2}
 \int |x|^k|\kf(x)| \d x< \infty.  \nonumber
\end{equation*}
Formula (\ref{Ordnung_Kern_1}) is equivalent to stating that $\left( \Fourier \kf\right)^{(m)}(0)$ vanishes for $1\leq m <k$.
\subsubsection{Rate results under global regularity assumptions}
We start by providing rate results under global regularity assumptions on the signed measure $\mu$ and on  the test function $f$, measured in a Sobolev sense.  This point of view is appropriate when one is interested in estimating integrals of the form $\theta= \int f(x) \mu(\d x)$ with some function $f$ which does not vanish near the origin.

Let us introduce the following  nonparametric classes of signed measures: 
\begin{definition}
\label{Definition_Glattheitsklasse_Masse_Sobolev} For constants $\bar{C}_1, \bar{C}_2, C_{\phi}, M_{\mu} >0$, 
 $ a \in \R $ and $c_{\phi},\beta, \rho \geq 0 $,
let 
$ \mathcal{M}(\bar{C}_1, \bar{C}_2, C_{\phi}, c_{\phi}, \beta, \rho, a, M_{\mu})$
be the collection of finite signed measures $\mu$ for which the following holds:
\begin{enumerate}
\item[(i)] There is a L\'evy process $X$ such that  the assumptions (A1)-(A4) are satisfied
and for the corresponding L\'evy measure $\nu$,  \mbox{$\mu(\d x) = x\nu(\d x)$}.
\item [(ii)] For the characteristic function $\phi$ of $X_1$, the following holds: 
\begin{eqnarray*}
\label{Definition_Glattheitsklasse_Abfall_chF}
 \forall u\in \R: \  |\phi(u)|\geq C_{\phi} (1+|u|^2)^{-\frac{\beta}{2}} e^{-c_{\phi} |u|^{\rho}}.
\end{eqnarray*}
\item[(iii)] For $C_1$ and $C_2$ defined  in (\ref{Definition_Konstante_1}) and (\ref{Definition_Konstante_2}), 
$C_1\leq \bar{C}_1$ and $C_2\leq \bar{C}_2$.
\item[(iv)] $\mu$ belongs to the Sobolev class $\mathcal{S}(a,M_{\mu})$.
\end{enumerate}
 \end{definition}
\hspace*{-0,3cm}We denote by  $\PP_{\mu}$ the distribution of $X_1$ and by $\E_{\mu}$ the expectation with respect to $\PP_{\mu}$.
\begin{theorem}
\label{Hauptsatz_Konvergenzraten_Sobolev}
Assume that \mbox{$f \in \mathcal{S}(s,M_f)$}  for some $s\in \R$ and some positive constant $M_f$.
Consider the nonparametric class $\mathcal{M}:= \mathcal{M}(\bar{C}_1, \bar{C}_2, C_{\phi}, c_{\phi}, \beta, \rho, a, M_{\mu})$,
with \mbox{$a>-s$}.   Assume  that $\Fourier \kf $ is supported on $[-\pi,\pi]$ and that either $\Fourier \kf=1_{[-\pi,\pi]}$  (sinc kernel) or $\kf$ has order $a+s$ and 
$\Fourier \kf \in \mathcal{H}_{\R}(a+s, L_{\kf}, R_{\kf})$ for positive constants $L_{\kf}$ and  $R_{\kf}$.

Then, selecting \mbox{$h_{\Delta,n}^*$} in an optimal way,  we derive  that
\begin{equation*}
 \label{Rate_Sobolev}
\sup_{\mu \in \mathcal{M}}  \E_{\mu} \left[ \left| \theta - \hat{\theta}_{h^{*}_{\Delta,n},\Delta, n}\right|^2 \right] = O\left( r_{\Delta,n} \right),
\end{equation*}
with $(r_{\Delta,n})$ denoting the sequences which are  summarized in the following table:
   \setlength{\tabcolsep}{10pt}
\renewcommand{\arraystretch}{2}
\begin{center}
\begin{tabular}{|l||l | c||l | c|}
\hline
\  &  \multicolumn{2}{c||} { {$\overline{C}_1<\infty$} } & \multicolumn{2}{c|}{  {$\overline{C}_1=\infty$} } \\
\hline
\hline
\multirow{3}{*} { {$\rho=0$} } & {$s\geq \Delta \beta$} & {$ T^{-1}$}  &  {$ s\geq \Delta \beta +\frac{1}{2}$} & {$T^{-1}$} \\
\cline{2-5}
& {$s=\Delta \beta$} &  {$T^{-1} $} &  {$ s =  \Delta \beta +\frac{1}{2}$} & {$(\log T) T^{-1}$}  \\
\cline{2-5}
& {$s < \Delta \beta$}  & {$ T^{-\frac{2a+2s}{2a+2 \Delta \beta}}$ } & {$s<\Delta \beta+\frac{1}{2}$} & {$n^{-\frac{2a+2s}{2a+2\Delta \beta+1}}$}\\
\hline
{$\rho>0$}  & \  &  {$(\frac{\log T}{\Delta})^{-\frac{2a+2s}{\rho}}$} & \  &  {$(\frac{\log T}{\Delta})^{-\frac{2a+2s}{\rho}}$} \\
\hline
\end{tabular}
\end{center}
\end{theorem}
\    \\
\textit{Discussion:} These rates of convergence should be compared to the  results which are known in a deconvolution framework, see \citet*{Butucea}. 

However, it is important to keep in mind that there are striking structural differences between the \Levy \ setting and a density deconvolution problem.

Very much unlike in a deconvolution framework, the parameters $a, \beta$ and $\rho$  are by no means independent of each other. A fast decay of the characteristic functions will always indicate a high activity of small jumps, so the jump measure cannot have a globally smooth Lebesgue density, but will be ill-behaved near the origin. Large values of $\beta$ and $\rho$ will hence not only in themselves lead to slow rates of convergence, but also result in small values of $a$. For a detailed discussion, we refer to \cite{Orey}, \cite{Belomestny_tcLevy} and \cite{Kappus}. 
\subsubsection{Rate results under local regularity assumptions}
If $f$ has a compact support which is bounded away from the origin,  the point of view of measuring the smoothness of $\mu$ in a global Sobolev sense is inappropriate.  This is true, for example, for point estimation and the estimation of derivatives. 
In this setting, we consider classes of measures which have a H\"older regular density in a neighbourhood of the point or interval of interest. 
\begin{definition}
\label{Definition_Glattheitsklasse}  For constants $\bar{C}_1, \bar{C}_2, C_{\phi}, L, R, a >0$ and  
$c_{\phi}, \beta, \rho \geq 0$  and a bounded open interval $D=(d_1,d_2)$, let $\mathcal{M}(\bar{C}_1, \bar{C}_2, C_{\phi}, c_{\phi}, \beta, \rho, a, D,L,R)$
be the collection of finite signed measures $\mu$, for which the following holds:
The items (i)-(iii) from Definition \ref{Definition_Glattheitsklasse_Masse_Sobolev} are true and
\begin{enumerate}
\item[(iv)]  The restriction $\mu|_D$ possesses a Lebesgue density  $g_D\in \mathcal{H}_D(a,L,R)$.
\end{enumerate}
 \end{definition}
\begin{theorem}
\label{Hauptsatz_Raten_Hoelder}
Assume that  $f$ is compactly supported, with $\supp(f)=:[a,b]\subseteq \R\setminus\{0\}$.
Assume, moreover, that for some $s\in \Z$ and some positive integer $C_f$, 
\begin{equation*}
\forall u\in \R:  |\Fourier f(-u)| \leq C_f (1+|u|^2)^{-s}.
\end{equation*}
Assume  that the order of $\kf$ is $a+s$ , $\kf$ is $-s\vee 0$-times continuously differentiable
and there is a constant $C_{\kf}>0$ such that for any nonnegative integer $m\leq (0\vee -s)$,
\begin{equation*}
 \forall z \in \R:    \left|\kf^{(m)}(z)\right|\leq C_{\kf} (1+|z|)^{-(a+s)-m-1}. 
 \end{equation*}
  Consider the nonparametric class
$\mathcal{M}= \mathcal{M}(\bar{C}_1, \bar{C}_2, C_{\phi}, c_{\phi}, \beta, \rho, a, D,L,R)$ with $D=(d_1,d_2)\supset [a,b]$ and $a>-s$.  Then, selecting $h_{\Delta,n}^*$ in an optimal way, we derive that 
\begin{equation*}
\sup_{\mu \in \mathcal{M}  }\E_{\mu} \left[  \left|  \theta -\hat{\theta}_{h^*_{\Delta,n}, \Delta,n}  \right|^2 \right] =
O \left(r_{\Delta,n} \right)
\end{equation*}
with the rates $r_{\Delta,n}$ collected in the following table:
   \setlength{\tabcolsep}{9pt}
\renewcommand{\arraystretch}{2}
\begin{center}
\begin{tabular}{|l||l | c||l | c|}
\hline
\  &  \multicolumn{2}{c||} { {$\overline{C}_1<\infty$} } & \multicolumn{2}{c|}{  {$\overline{C}_1=\infty$} } \\
\hline
\hline
\multirow{3}{*} { {$\rho=0$} } & {$s>\Delta \beta +\frac{1}{2}$} & {$ T^{-1}$}  &  {$ s > \Delta \beta + 1$} & {$T^{-1}$} \\
\cline{2-5}
& {$s = \Delta  \beta +\frac{1}{2}$}  & {$ \left(\log T\right) T^{-1}$ } & {$s= \Delta \beta +1 $} & {$(\log T) T^{-1}$}\\
\cline{2-5}
& {$s  <  \Delta \beta +\frac{1}{2}$}  & {$  T^{-\frac{2s+2a}{2\Delta \beta +2a+1} }$ } & {$s<\Delta \beta +1$} & {$T^{-\frac{2a+2s}{2a+2\Delta \beta+2}}$}\\
\hline
{$\rho>0$}  & \  &  {$\left(\frac{\log T}{\Delta}\right)^{-\frac{2a+2s}{\rho}} $} & \  &
{$\left(\frac{\log T}{\Delta}\right)^{-\frac{2a+2s}{\rho}}$} \\
\hline
\end{tabular}
\end{center}
\end{theorem}
\   \\
\textit{Diskussion:} For point estimation (s=0) or the estimation of the k-th derivative (s=-k), we recover the rates of convergence which are classical and known to be optimal  in a density deconvolution framework, see \citet{Fan}. 

In the continuous limit,  for $\Delta$ close to zero, we recover the rates which are typical for density estimation with pointwise loss.  The rates of convergence found for pointwise  loss should also be compared to the results found  in \citet{Belomestny_tcLevy}. 

When the estimation of  $\mu(A)= \int 1_A \mu(\d x)$ for some $A$ bounded away from zero is being considered, we have $s=1$.

For the particular cases of estimating integrals and point estimation, the rates of convergence  are known to be  minimax optimal, see  \citet{Kappus_Diss} for lower bound results. 
%%%%%%%%%%%%%%%%%%%%%
%%%%%%%%%%%%%% Adaptive Schaetzung 
%%%%%%%%%%%%%%%%%%%%%%%%
\section{Adaptive estimation}\label{Abschnitt_adaptive_Schaetzung}
Let a finite collection \mbox{$\mathcal{M}=\{m_1, \cdots, m_n\}\subseteq \mathbb{N}$} of  indices be given and let $\mathcal{H}:= \{1/m_1, \cdots, 1/m_n\}$ be a collection of bandwidths associated  with  $\mathcal{M}$.

For notational convenience, we suppress,  in this section, the dependence on  $\Delta$ and assume  $\Delta=1$. Moreover, we slightly change the notation and denote the kernel estimator by $\hat{\theta}_{m,n}$ instead of 
$\hat{\theta}_{1,1/m,n}$. 

The goal of this section is to provide a strategy for the optimal data driven choice of the smoothing parameter 
$\hat{m}$ within the collection
$\mathcal{M}_n$ and to derive, for the corresponding estimator $\hat{\theta}_{\hat{m}, n}$, the oracle  inequality 
\begin{equation}
\label{Orakelungleichung}
 \E \left[|\theta - \hat{\theta}_{\hat{m},n} |^2   \right] 
 \leq 
C^{ad} \inf_{m\in \mathcal{M}}\bigg\{ |\theta- \theta_{m_n}|^2   + \underset{k\in \mathcal{M}}{\sup_{ k\geq m}}|\theta_k - \theta_m|^2
+\pen(m)    \bigg\} + C n^{-1},
\end{equation}
with
\begin{equation*}
\label{Definition_Projektion_Funk_Levy}
\theta_m := \int f(x) \left( \kf_{\frac{1}{m}}\ast \mu\right)(x) \d x,
\end{equation*}
with positive constants $C^{ad}$ and $C$ which  do not depend on the unknown underlying smoothness parameters and a penalty term $\pen(m)$ to be
specified, which equals, up to some logarithmic factor, the quantity
\begin{equation*}
\frac{1}{n} \sigma_m^2 :=  \frac{n^{-1}}{2 \pi^2} \bigg\{C_1
\int \left|\frac{\Fourier f(-u)\Fourier \kf \left( u/m \right)}{\phi(u)}\right|^2  \d u 
\wedge C_2\left( \int   \left|\frac{\Fourier f(-u)\kf \left( u/m\right)}{\phi(u)}\right|  \d u \right)^2
\bigg\},
\end{equation*}
which bounds the error in the model.

The occurrence of the supremum  term in the   oracle-inequality  (\ref{Orakelungleichung}) is typical for  the adaptive estimation of linear functionals, see \citet*{Cai_Low_1} for lower bound results.\\[0,2cm]
In a deconvolution framework with known distribution of the noise, see \citet*{Butucea},  the way to go is to estimate the quantities $|\theta_k - \theta_m|^2$ involved in the oracle bound
by their bias-corrected version, that  is, to consider $|\hat{\theta}_k -\hat{\theta}_m|^2 - \strafH^2(m,k) $, with
some deterministic correction term $\strafH^2(m,k)$ which is chosen large enough 
 to ensure that with high probability,
\[
 |\hat{\theta}_k - \hat{\theta}_m|^2 - \strafH^2(m,k) \leq |\theta_k - \theta_m|^2  \   \forall m,k \in \mathcal{M}.
\]
On the other hand, $\strafH^2(m,k)$ should ideally not be much larger than the variance term.

The appropriate choice then turns out to be 
\[
\strafH^2(m,k):= \frac{1}{n} \rho \lambda_{m,k}^2 \left( \sigma_{m,k}^2  + x_{m,k}^2 \right)
\]
with some positive constant $\rho$ to be appropriately chosen and
\begin{eqnarray*}
  \sigma_{m,k}^2   & := &  \frac{1}{2\pi^2} \Bigg\{
C_1 \int \left| \frac{\Fourier f(-u)}{\phi(u)}\right|^2 \left| \kerndiffmk \right|^2 \d u  \\ 
& \phantom{:=} & \phantom{\frac{1}{2\pi^2} \Bigg\{}
 \wedge \,
C_2  \left( \int  \left| \frac{\Fourier f(-u)}{\phi(u)}\right| \left| \kerndiffmk \right| \d u \right)^2  \Bigg\}
\end{eqnarray*}
and
\begin{eqnarray*}
 x_{m,k}:=\frac{1}{\sqrt{n}}\frac{1}{2\pi}  \int \left|\frac{\Fourier f(-u)}{\phi(u)}  \right| \left|\kerndiffmk  \right| \d u
\end{eqnarray*}
and with logarithmic weights $\lambda_{m,k}$ chosen large enough to ensure  \hspace*{-0,5cm} $\sum_{k\in\mathcal{M}, k>m}\limits \hspace*{-0,5cm} e^{-\lambda_{m,k} }<\infty$.  See\\[-0,2cm] also 
 \citet*{Prieur} for the underlying ideas. 

Indeed, this is the fundamental idea about model selection via penalization: Deterministic terms are applied in order
to control the fluctuation of certain stochastic quantities, uniformly over some countable index set.  In a white noise framework, this principle is illustrated in  \citet{Birge}.

The situation is different in the present framework, since the optimal theoretical penalty is no longer feasible, but depends on the characteristic function in the denominator, which is unknown. 

It is intuitive to work with a stochastic penalty term and replace the unknown characteristic function by its empirical version. The model selection procedure will then crucially  depend on a precise control of the fluctuation of the empirical characteristic function in the denominator around its target.

In a similar setting, \citet*{Comte_Genon_1} have dealt with this problem by proposing an additional a priori assumption on the size of the collection $\mathcal{M}$.  However, this assumption is critical and highly restrictive, since it depends on the decay behaviour of the unknown characteristic function and hence involves some prior knowledge of the underlying smoothness parameters.  

\citet*{Comte_Lacour} have proposed another approach towards model selection with unknown variance, which does not depend on any
prior knowledge of the smoothness parameters. However, this approach is designed for $\lk^2$-loss and spectral cutoff estimation and
the generalization to the estimation of linear functionals  with general kernels is not straightforward. Moreover,  the strategy proposed in that paper would lead, in the present case, to a polynomial loss. For this reason we propose a different strategy, which will also allow to drop certain semi-parametric assumptions on $\phi$. 

In what follows, we introduce a  newly defined estimator of the characteristic function in the denominator. The fluctuation of this object can be controlled not only pointwise, but uniformly on the real line, which is the key to making the model selection procedure work under very weak assumptions. 
\begin{definition}
\label{Definition_Neumann_alternativ_2}
Let the weight function $w$ be 
\begin{equation*}
w(u) = \left( \log \left( e+|u| \right) \right)^{-1/2-\delta}
\end{equation*}
for some $\delta>0$. For a constant $\kappa$ to be chosen, let the truncated version of 
$\hat{\phi}_n(u)$ be 
\begin{equation*}
\label{Definition_phi_abgeschnitten}
{\check{\phi}}_n(u) := \begin{cases}
     \hat{\phi}_n(u), & \hspace*{-2cm}\textnormal{if}  \    |\hat{\phi}_n(u)| \geq \kappa (\log n)^{1/2} w(u)^{-1} n^{-1/2}   \\
    \kappa (\log n)^{1/2} w(u)^{-1} n^{-1/2}, & \textnormal{else}.
\end{cases}
\end{equation*}
Let the corresponding estimator of ${1}/{\phi}$ be  $ 1/\check{\phi}_n$.
\end{definition}
The definition of the weight function originates from \citet*{Neumann_Reiss} and the considerations presented therein will play an important role for our arguments. Introducing the extra  logarithmic factor in the definition of ${1}/{\check{\phi}}$   will enable us to apply concentration inequalities of Talagrand type.
This is  the key to proving 
 the following uniform version of the pointwise result which has been stated in Lemma 2.1 in  \citet{Neumann}:

\begin{lemma}
\label{Neumann_Lemma_uniform} Let $c_1$ be the constant appearing in Talagrand's inequality (see Lemma
\ref{Talagrand-Ungleichung}).
Let   $\kappa$ be  chosen such that for some
$\gamma >0$,  $\kappa \geq 2(\sqrt{2 c_1} +\gamma)$. Then we have for some constant $C\neukap$ depending on the choice of
 $\kappa, \gamma$ and
$\delta$,
\begin{eqnarray*}
\E \left[ \sup_{u\in \R}\frac{  \left|\phinneuu -\phinu  \right|^2  }{ \frac{(\log n) w(u)^{-2} n^{-1}}{|\phi(u)|^4 }
\wedge \frac{1}{|\phi(u)|^2}}    \right]
\leq C\neukap.
\end{eqnarray*}
\end{lemma}
The above definition gives rise to the following redefinition of the kernel estimator: In what follows, we set 
\begin{equation*}
\label{Definition_Kernschaetzer_neu_Levy}
\hat{\theta}_{m,n}:=\frac{1}{2\pi} \int \Fourier f(-u) \frac{\hat{\phi}_n'(u)}{i {\check{\phi}_n}(u)} \Fourier \kf \left(u/m \right) \d u.
\end{equation*}For $m,k \in \mathcal{M}$, we define the stochastic correction term 
\begin{equation*}
\label{Definition_Strafterm_alternativ}
 \widetilde{\strafH}^2(m,k):= n^{-1}   \left\{ c^{\pen} c_1 \tilde{\lambda}_{m,k}^2
       + 16 \left( \frac{5}{2}  \kappa \right)^2 (\log n)  \right\}  \left(  \tilde{\sigma}_{m,k}^2  \vee \tilde{x}^2_{m,k}  \right)
\end{equation*}
with
\begin{eqnarray*}
\label{Definition_sigma_stochastisch}
\notag &&\hspace*{-1,5cm}\tilde{ \sigma}_{m,k}^2 : =    \frac{1}{2\pi^2} \Bigg\{ \overline{C}_1 \int  \left|\frac{\Fourier f(-u)}{{\check{\phi}}_n(u) }\right|^2  \left|\Kerndifmk \right|^2 w(u)^{-2} \d u   \\
&& \hspace*{-1,5cm} \phantom{\tilde{\sigma}_{m,k}^2 =\frac{1}{2\pi^2} \Bigg\}  }
\wedge \, \overline{C}_2 \left(\int \left|\frac{\Fourier f(-u)}{{\check{\phi}}_n(u) }\right| \left|\Kerndifmk   \right| w(u)^{-1} \d u \right)^2 \hspace*{-0,1cm} \Bigg\}
\end{eqnarray*}
for positive constants $\overline{C}_1$ and $\overline{C}_2$ and with 
\begin{equation*}
\label{Definition_x_stochastisch}
\tilde{x}_{m,k} := \frac{1}{\sqrt{n} } \frac{1}{2 \pi}\int \left| \frac{\Fourier f(-u)}
{{\check{\phi}}_n(u)}\right| \left|\Kerndifmk \right| w(u)^{-1}
\d u.
\end{equation*}
Let the weights be defined as follows: For some $\eta>0$,  
\begin{eqnarray*}
\label{Definition_Gewichte_Levy}
\notag &&\hspace*{-0,8cm}\tilde{\lambda}_{m,k} :=  \frac{8}{\eta}  \log \left(   \log (n \tilde{x}_{m,k} (k-m) )\right)^2
\log( n \tilde{x}_{m,k} (k-m) )  \log \left(  \tilde{x}_{m,k}^2 (k-m)^2      \right)   \\
     &&\hspace*{-0,8cm}  \phantom{\lambda_{m,k}:=}             \vee   \log \left(\tilde{\sigma}_{m,k}^2 (k-m)^2 \right) .
\end{eqnarray*}
For some $\gamma>0$, let   $c^{\pen} =  64 \vee 16(2c_1+\gamma)$ and
and  $\kappa = 2(\sqrt{4 c_1}+\gamma)$. 
Finally, let
\begin{equation*}
\label{Definition_pen_alternativ}
 \tilde{\pen}(m):= \tilde{\strafH}^2(0,m).
\end{equation*}
We understand by $\pen(m)$ and $\strafH^2(m,k)$ the deterministic versions of $\tilde{\pen}(m)$ and $\tilde{\strafH}^2(m,k)$, that
is, the definitions are the same as in formula (\ref{Definition_Strafterm_alternativ}) and formula
(\ref{Definition_pen_alternativ}), apart from the fact that $
{1}/{\check{\phi}_n}$ is replaced by ${1}/{\phi}$.

These definitions give rise to the following choice of the cutoff parameter:
\begin{equation*}
 \label{empirischer_cutoff_endgueltig}
 \hat{m}  :=  \arginf_{m\in \mathcal{M}}  \limits  {\sup_{k>m,k\in \mathcal{M}}}\left\{ |\hat{\theta}_k -\hat{\theta}_m|^2 -\tilde{\strafH}^2(m,k)  \right\}  + \tilde{\pen}(m).
\end{equation*}
We are now ready to state the following oracle bound and hence the main result of this section: 
\begin{theorem}
\label{Hauptsatz_adaptiver_Schaetzer_Levy} Assume that
 (A1)- (A4),  (F1) or (F2) and (K1)+(K2) hold.  Assume that $C_{1}\leq \bar{C}_1$ and $C_2\leq \bar{C}_2$.
Assume, moreover, that $\E\left[\exp\left(\eta |X_1|\right)   \right]  <\infty$.   Then 
\begin{eqnarray*}
\E \left[|\theta- \hat{{\theta}}_{\hat{m}}  |^2  \right] \leq  C^{ad} \inf_{m\in \mathcal{M}}\limits
\big\{ |\theta -\theta_{m_n}|^2  +
\underset{k\in \mathcal{M}}{ \sup_{k> m}} |\theta_k-\theta_m|^2  +\pen(m)    \big\}   + C n^{-1}   
\end{eqnarray*}
holds with some  $C^{ad}>0$  depending   on the particular choices of the constants, but not on the unknown parameters, and with some  
$C>0$ depending on the choice of the constants and on $\eta^{-1}\E\left[\exp\left(\eta |X_1|\right)\right]$.
\end{theorem}
It is worth mentioning that the exponential moment assumption can be relaxed, but at the cost of losing a polynomial factor and complicating the proofs. For sake of simplicity, we omit the details. 

In comparison with the adaptive results obtained in \citet*{Comte_Lacour} or \citet*{Schwarz_Johannes}, it is remarkable that our procedure is completely model free in the sense that it no longer depends on any semiparametric assumption on the decay of the characteristic function, be it exponential or polynomial decay or the assumption that $\phi$ can, up to some constant, be bounded from above and below by some monotoneously decreasing function. Moreover, apart from the condition that the linear functional is well defined, the measure $\mu$ may be fairly arbitrary and is not assumed to belong to any prescribed nonparametric class. 

From an asymptotic point of view, the above results tells us that the procedure attains, up to a logarithmic loss, the minimax rates of convergence for the particular cases of point estimation, estimating derivatives and for estimating integrals. 

It is well known that for estimating linear functionals, the loss of a logarithmic factor due to adaptation can typically not be avoided, see \citet{Lepski_3}. Consequently, it is not surprising that a logarithmic loss is also found in the L\'evy model. 
%%%%%%%%%%%%%%%%%%%%
%    Beweisteil   
%%%%%%%%%%%%%%%%%
\section{Proofs}
\label{Beweisteil}
\subsection{Proofs of the main results of Section \ref{Abschnitt_Schaetzer_und_Raten}}
\subsubsection{Risk bounds}
The proof of Theorem \ref{Hauptsatz_allgemeiner_Kernschaetzer_Levy} essentially relies on the following auxiliary result:
\begin{lemma}
\label{Hauptabschaetzung} \
For some universal positive constant $C$,
\begin{align*}
\notag \phantom{=} &\,  \left| \E\left[ \left( \frac{\frac{1}{\Delta} \hat{\phi}_{{\Delta,n}}'(u)}{\tilde{\phi}_{{\Delta,n}}(u)}- \frac{\frac{1}{\Delta}\phi_{\Delta}'(u)}{\phi_{\Delta}(u)}\right) 
\left( \frac{\frac{1}{\Delta}\hat{\phi}_{{\Delta,n}}'(\text{-}v)}{\tilde{\phi}_{{\Delta,n}}(\text{-}v)}-
\frac{\frac{1}{\Delta}\phi_{\Delta}'(\text{-}v)}{\phi_{\Delta}(\text{-}v)}\right) \right] \right|\\
   \leq &\,   C  \left( \frac{T^{\text{-}1}}{|\phi_{\Delta}(u)\phi_{\Delta}(\text{-}v)|}\wedge 1 \right)
\left( |\Psi''(u-v)|+ |\Psi'(u-v)|^2+ |\Psi'(u)\Psi'(\text{-}v)|\right).
\end{align*}
\end{lemma}
\begin{proof}
We start by noticing that  for some constant $C\neu_k$ depending only on $k$,
\begin{equation}
\label{Neumanns_Lemma_Delta}
\E \left[ \left|  \frac{1}{\tilde{\phi}_{{\Delta,n}}(u) }- \frac{1}{\phi_{\Delta}(u)}  \right|^{k} \right]
\leq C\neu_k \left( \frac{T^{-\frac{k}{2}}}{|\phi_{\Delta}(u)|^{2k}} \wedge \frac{1}{|\phi_{\Delta}(u)|^k} \right).
\end{equation}
 This is a direct consequence of Lemma 2.1  in \cite{Neumann}, tracing  back the dependence \mbox{on $\Delta$}.

In the sequel, let $\text{R}_{{\Delta, n}}(u):= \frac{1}{\tilde{\phi}_{\Delta,n}(u) }- \frac{1}{\phi(u)}$. Using the fact that $\hat{\phi}_{{\Delta, n}}'$ and ${1}/{\tilde{\phi}_{{\Delta,n}}}$ are independent by construction and that 
$\hat{\phi}_{{\Delta,n}}'(u)- \phi_{\Delta}'(u)$
is  centred, we  find that
\begin{align*}
\phantom{=} & \E\left[ \hspace*{-0,5mm}\left(\frac{\hat{\phi}_{{\Delta,n}}'(u)}{\tilde{\phi}_{{\Delta,n}}(u)} -\frac{\phi_{\Delta}'(u)}{\phi_{\Delta}(u)} \right)\hspace*{-0,9mm}
\left(\frac{\hat{\phi}_{{\Delta,n}}'(\text{-}v)}{\tilde{\phi}_{{\Delta,n}}(\text{-}v)} -
\frac{\phi_{\Delta}'(\text{-}v)}{\phi_{\Delta}(\text{-}v)} \right) \hspace*{-0,5mm}  \right]  \\
 {=} &  \E \left[ \hspace*{-0,5mm}\left(  \frac{ (\hat{\phi}_{{\Delta,n}}'(u) - \phi_{\Delta}'(u))}
{\tilde{\phi}_{{\Delta, n}}(u)} +  \phi_{\Delta}'(u) \text{R}_{\Delta,n}(u)  \right) \hspace*{-0,9mm}
 \left(  \frac{ (\hat{\phi}_{{\Delta,n}}'(\text{-}v) - \phi_{\Delta}'(\text{-}v))}
{\tilde{\phi}_{{\Delta, n}}(\text{-}v)} +  \phi_{\Delta}'(\text{-}v)\text{R}_{\Delta,n}(\text{-}v) \right)  \hspace*{-0,5mm} \right]\\
=&\Cov(\hat{\phi}_{{\Delta,n}}'(u),\hat{\phi}_{{\Delta,n}}'(v) ) 
\E\hspace*{-0,05cm}\left[\frac{1}{\tilde{\phi}_{{\Delta,n}}(u)  \tilde{\phi}_{{\Delta,n}}(\text{-}v)}  \right] + \phi_{\Delta}'(u)\phi_{\Delta}'(\text{-}v) \E\left[\text{R}_{{\Delta, n}}(u)\text{R}_{{\Delta, n}}(\text{-}v)  \right].
\end{align*}
The Cauchy-Schwarz-inequality and then an application of (\ref{Neumanns_Lemma_Delta}) imply
\begin{align}
\notag    \E\left[\left|\text{R}_{{\Delta, n}}(u) \text{R}_{{\Delta, n}}(\text{-}v) \right|\right]  
 \leq &\,  \left( \E\left[ \left| \text{R}_{{\Delta, n}}(u) \right|^2 \right] \right)^{\frac{1}{2} }  \left( \E\left[\left|\text{R}_{{\Delta, n}}(\text{-}v)  \right|^2  \right] \right)^{\frac{1}{2}}\\
\label{LOS_DZ_3}  \leq &\,   C\neu_2 \left(\frac{ T^{-1} }{|\phi_{\Delta}(u)|^2 |\phi_{\Delta}(\text{-}v)|^2}  
\wedge \frac{1}{|\phi_{\Delta}(u)||\phi_{\Delta}(\text{-}v)|} \right).
\end{align}
Using the triangle inequality, again (\ref{Neumanns_Lemma_Delta}) and then (\ref{LOS_DZ_3}), we find that
\begin{align*}
 \phantom{\leq}&\,  \E \left[ \left|\frac{1}{\tilde{\phi}_{{\Delta, n}}(u) \tilde{\phi}_{{\Delta, n}}(\text{-}v)}  \right| \right] \leq  \E \left[\left(|\text{R}_{{\Delta, n}}(u)|+\frac{1}{|\phi_{\Delta}(u)|}\right)
 \left(|\text{R}_{{\Delta, n}}(\text{-}v)|+\frac{1}{|\phi_{\Delta}(\text{-}v)|} \right)   \right]\\
\leq &\,  (1+ 2 C\neu_1 + C\neu_2 ) \frac{1}{|\phi_{\Delta}(u)||\phi_{\Delta}(\text{-}v)|}.
\end{align*}
On the other hand, by definition of $\frac{1}{\tilde{\phi}_n}$, 
\begin{equation*}
 \E \left[\left|\frac{1}{\tilde{\phi}_{{\Delta, n}}(u) \tilde{\phi}_{{\Delta, n}}(\text{-}v)}  \right| \right] \leq (\Delta n)= T.
\end{equation*}
Next, 
\begin{align}
\notag  \phantom{=} &\, \left|  \Cov (\hat{\phi}_{{\Delta, n}}'(u) , \hat{\phi}_{{\Delta, n}}'(v))  \right| 
 =  n^{\text{-}1} \left|\E \left[ (iZ_{\Delta})^2 e^{i(u-v) Z_{\Delta}}\right] - \E \left[iZ_{\Delta} e^{iuZ_{\Delta}} \right]\E \left[iZ_{\Delta} e^{\text{-}ivZ_{\Delta}} \right]\right|  \\
\notag   =&\,  n^{\text{-}1} \left| \phi_{\Delta}''(u-v) - \phi_{\Delta}'(u) \phi_{\Delta}'(\text{-}v)  \right| \\
 \label{tagtag}\leq &\,   n^{\text{-}1} \left( |\Delta \Psi''(u-v)| +|\Delta \Psi'(u-v)|^2 +\Delta^2|\Psi'(u) \Psi'(\text{-}v)| \right).
\end{align}
Putting (\ref{LOS_DZ_3})-(\ref{tagtag}) together, we have 
shown
\begin{align*}
\notag 
\phantom{\leq} &\,   \left|\Cov\Big(\hat{\phi}_{{\Delta, n}}'(u),\hat{\phi}_{{\Delta, n}}'(v)\Big)\right|
 \left| \E \left[\frac{1}{\tilde{\phi}_{{\Delta, n}}(u) \tilde{\phi}_{{\Delta, n}}(\text{-}v)}   \right] \right| \\
\label{Lemma_obere_Schranken_siebte_Zerlegung}\leq &\,   C' \left(  \frac{ T^{-1}}{|\phi_{\Delta}(u)||\phi_{\Delta}(\text{-}v)|}\wedge 1 \right) \Delta^2\left(\left| \Psi''(u-v)\right| +| \Psi'(u-v)|^2  +\left| \Psi'(u) \Psi'(\text{-}v)\right|\right). 
\end{align*}
With constant $C':= 1+2C\neu_2+C\neu_1$.

Another  application of   (\ref{LOS_DZ_3}), gives 
\begin{align}
\notag \phantom{\leq} &\,    \left|\E\left[\text{R}_{{\Delta, n}}(u)
 \text{R}_{{\Delta, n}}(\text{-}v)  \right]\right|| \phi_{\Delta}'(u) \phi_{\Delta}'(\text{-}v)|   =   \left|\E\left[\text{R}_{{\Delta, n}}(u)
 \text{R}_{{\Delta, n}}(\text{-}v)  \right]\right|\Delta^2 |\phi(u)\phi(\text{-}v ) 
\Psi'(u) \Psi'(\text{-}v)|\\
\notag \leq&\,   C\neu_2 \left(\frac{T^{-1}}{|\phi_{\Delta}(u)||\phi_{\Delta}(\text{-}v)|} \wedge 1 \right)
\Delta^2 \left|\Psi'(u)  \right|  \left|\Psi'(\text{-}v)\right|.
\end{align}
This completes the proof.
\end{proof}
\begin{proof}[\textbf{Proof of Theorem  \ref{Hauptsatz_allgemeiner_Kernschaetzer_Levy} }]   Given the kernel function $\kf$ and bandwidth $h$, let  \mbox{$\theta_{h}:= \int f(x) (\kf_h\ast \mu)(x) \d x$}.  We can estimate
\begin{eqnarray}
\notag &\phantom{\leq}&    \E \left[\left|\theta - \hat{\theta}_{\Delta, h,n}    \right|^2  \right] \leq  2 \left|\theta - \theta_h  \right|^2  + 2\E \left[\left|\theta_h - \hat{\theta}_{\Delta, h,n}  \right|^2   \right].
\end{eqnarray}
In what follows, we use the notation
\[
T(u,v):=   \left( \frac{\hat{\phi}'_{{\Delta, n}}(u) }{\tilde{\phi}_{{\Delta, n}}(u)} - \frac{\phi'_{\Delta}(u)}{\phi_{\Delta}(u)}\right)
\left(\frac{\hat{\phi}'_{{\Delta, n}}(\text{-}v)}{\tilde{\phi}_{{\Delta, n}}(\text{-}v) }- \frac{\phi'_{\Delta}(\text{-}v)}
{\phi_{\Delta}(\text{-}v)} \right). 
\]
Passing to the Fourier domain and applying Fubini's theorem yields 
\begin{align}
\notag \phantom{=} &  \E \left[ \left|\theta -\hat{\theta}_{\Delta, h,n} \right|^2\right] =\E \left[ \left|\frac{1}{2\pi}\int \Fourier f(-u) \Fourier \kf_h(u) \frac{1}{\Delta} \left(\frac{\hat{\phi}'_{{\Delta,n}}(u)}{\tilde{\phi}_{{\Delta,n}}(u)}-
\frac{\phi'_{\Delta}(u)}{\phi_{\Delta}(u)} \right) \d u  \right|^2  \right]\\
\notag =& \frac{1}{4\pi^2\Delta^2}\int \int \Fourier f(\text{-}u) \Fourier f(v) \Fourier \kf_h(u) \Fourier \kf_h(\text{-}v) \E \left[ T(u,v)   \right]\d u \d v.
\end{align}
Thanks to Lemma  \ref{Hauptabschaetzung},
\begin{align*}
 \phantom{\leq } &\,  \int \int   \Fourier f(\text{-}u) \Fourier f(v) \Fourier \kf_h(u) \Fourier \kf_h(-v)   \frac{1}{\Delta^2}
\E \left[ T(u,v)  \right] \d u \d v  \\
\notag \leq &\,  C T^{-1} \Bigg( \int \int\hspace*{-0,1cm} \frac{ \left|\Fourier f(\text{-}u) \Fourier f(v)  \right|}{|\phi_{\Delta}(u)\phi_{\Delta}(\text{-}v)|}
\left|\kf_h(u) \kf_h(\text{-}v)  \right|
 \left(\left| \Psi''(u-v) \right|\hspace{-0,05cm}+\hspace{-0,05cm}\left| \Psi'(u-v) \right|^2 \right)\hspace{-0,05cm}\d u \d v \\
\phantom{\leq}&\phantom{C T^{\text{-}1}\Bigg(    }+\int \int  \frac{|\Fourier f(\text{-}u)\Fourier f(v)|}{ |\phi_{\Delta}(u)\phi_{\Delta}(\text{-}v)|} \left|\kf_h(u) \kf_h(\text{-}v)  \right|
\left| \Psi'(u)  \Psi'(\text{-}v) \right|   \d u \d v\Bigg).
\end{align*}

If   $\Psi'' \in  \lk^1(\R)$ and $\Psi' \in \lk^2(\R)$, we apply the Cauchy-Schwarz inequality and Fubini's theorem to
find 
\begin{align}
\notag \phantom{\leq}&
\int \int \frac{ \left|\Fourier f(\text{-}u)  \Fourier f(v)  \right|}{|\phi_{\Delta}(u)\phi_{\Delta}(\text{-}v)|}  \left|\kf_h(u)  \kf_h(\text{-}v)  \right|
\left(  \left| \Psi''(u-v) \right| +\left|\Psi'(u-v)\right|^2 \right) \d u \d v  \\
\notag \leq  & \int \frac{\left| \Fourier f(\text{-}u) \right|^2}{|\phi_{\Delta}(u)|^2} \left|\Fourier \kf_h(u)  \right|^{2}
\int  (|\Psi''(u-v)| +|\Psi'(u-v)|^2 ) \d v \d u\\
 \notag  =& \left( \|\Psi''\|_{\lk^1}+\|\Psi'\|_{\lk^2}^2  \right)
\int  \frac{\left| \Fourier f(\text{-}u) \right|^2}{|\phi_{\Delta}(u)|^2} \left|\Fourier \kf_h(u)  \right|^{2}  \d u
\end{align}
and 
\begin{align*}
\int \int \frac{|\Fourier f(\text{-}u)\Fourier f(v) |}{|\phi_{\Delta}(u)\phi_{\Delta}(\text{-}v)|}
\left|\kf_h(u)  \kf_h(\text{-}v)  \Psi'(u)\Psi'(\text{-}v)\right|  \d u \d v 
   \leq    \|\Psi'\|_{\lk^2}^2  \int \frac{|\Fourier f(\text{-}u)\Fourier \kf_h(u)|^2}{|\phi_{\Delta}(u)|^2}  \d u.
\end{align*}
On the other hand, if $\Psi''\in \lk^1(\R)$ or $\Psi'\in \lk^2(\R)$ fails to hold,  we can always use the  estimate
\begin{align}
 \notag \phantom{\leq}&\,  
\int \int \frac{ \left|\Fourier f(\text{-}u)  \Fourier f(v)  \right|}{|\phi_{\Delta}(u)\phi_{\Delta}(\text{-}v)|}
\left|\kf_h(u) \kf_h(\text{-}v)  \right|
 ( \left| \Psi''(u-v) \right| +|\Psi'(u-v)|^2)  \d u \d v  \\
\notag \leq &\,   \left( \|\Psi''\|_{\infty} +\|\Psi'\|_{\infty}^2\right)\left(\int  \frac{|\Fourier f(\text{-}u)|}
{|\phi_{\Delta}(u)|} |\Fourier \kf_h(u)|\d u  \right)^2
\end{align}
and
\begin{align}
\notag &\phantom{\leq} \
\int \int \frac{|\Fourier f(\text{-}u)\Fourier f(v)|}{|\phi_{\Delta}(u)\phi_{\Delta}(\text{-}v)|} |\Fourier \kf_h(u)\Fourier \kf_h(\text{-}v)|
|\Psi'(u)\Psi'(\text{-}v)|\d u \d v  \\
\notag & = \
\left(\int \frac{\left|\Fourier f(\text{-}u)\right|}{|\phi_{\Delta}(u)|} |\Fourier \kf_h(u)| |\Psi'(u)| \d u \right)^2  \leq 
\left\|\Psi'\right\|_{\infty}^2
 \left( \int \frac{\left|\Fourier f(\text{-}u)\right|}{|\phi_{\Delta}(u)|} |\Fourier \kf_h(u)| \d u \right)^2.
\end{align}
This completes the proof. 
\end{proof}
\subsubsection{{Rates of convergence} }
\begin{lemma} \label{Lemma_Bias_Sobolev} In the situation of Theorem \ref{Hauptsatz_Konvergenzraten_Sobolev}, the approximation error can be  estimated as follows:
\begin{equation}
 \notag 
 \left| \int f(x) \mu(\d x) - \int f(x) (\kf_h\ast \mu)(x) \d x \right|^2 \leq \frac{C_B}{(2\pi)^2} h^{2a+2s} =: b_h
\end{equation}
with constant
\begin{equation*}
 C_B= \left(2 \pi^{-s-a} +\left(   \frac{L_K}{\langle a+ s\rangle !}\right) \right)^2 M_f M_{\mu}.
\end{equation*}
\end{lemma}
\begin{proof}
 By assumption, $f$ is Sobolev-regular with index $s$ and $\mu$ is Sobolev-regular with index $a>-s$. This implies, by duality of Sobolev spaces, that we can pass to 
the Fourier domain and write
\begin{equation}
\label{Darstellung_Bias_Sobolev}
\left| \int f(x) \mu(\d x) - \int f(x) (\kf_h\ast \mu)(x) \d x  \right|^2 = \left|\frac{1}{2\pi} \int \Fourier f(\text{-}u) (1-\Fourier \kf_h(u)) \Fourier \mu(u) \d u  \right|^2.
\end{equation}
Applying  the Cauchy Schwarz inequality and then the regularity assumptions on $f$ and on $\mu$, we find that
\begin{align}
\notag   \phantom{ =} &\, \left|\frac{1}{2\pi} \int \Fourier f(-u) (1-\Fourier \kf_h(u)) \Fourier \mu(u) \d u  \right|^2 \\
\notag  \leq &\, \frac{1}{(2\pi)^2} \int  |\Fourier f(-u)|^2 (1+|u|^2)^{s} \d u \int |1-\Fourier \kf_h(u)|^2 (1+|u|^2)^{-a-s} |\Fourier \mu(u)|^2 (1+|u|^2)^{a}\d u \\  
\label{Darstellung_Bias_Sobolev_2.5.}\leq &\,  \frac{M_f M_{\mu}}{(2\pi)^2} \,  \sup_{u\in \R}  |1-\Fourier \kf_h(u)|^2 (1+|u|^2)^{-a-s} .\end{align}
If $\kf$ is the sinc kernel, we can immediately  estimate
\begin{eqnarray}
\label{Bias_Sob-Sinc}
 \hspace*{-0,1cm} \sup_{u\in \R} |1- \Fourier \kf(hu)|^2 (1+|u|^2)^{-a-s}
 =   \sup_{|u|\geq \frac{\pi}{h} } (1+|u|^2)^{-a-s}
 \leq  \pi^{-2a -2s} h^{2a+2s},
\end{eqnarray}
which gives the desired result.

If $\Fourier \kf$ is $\langle a+s \rangle$-times continuously differentiable and the derivatives up to order
$\langle a+s \rangle  $ at zero vanish, a Taylor series expansion  gives for some \mbox{$\tau\in [-hu,hu]$}:
\begin{equation}
\label{Abschaetzung_Hoelderkern_1}
  1- \Fourier \kf(hu)  
 = \frac{\Fourier \kf^{\left(\langle a+s \rangle\right) }(\tau)}{\langle a+s \rangle !}(hu)^{\langle a+s \rangle }.
\end{equation}
By H\"older continuity of $\Fourier \kf^{\langle a+s\rangle}$, 
\begin{eqnarray}
 \label{Abschaetzung_Hoelderkern_2}
 \left| \Fourier \kf^{\left( \langle a+s \rangle \right)}(\tau) \right|
  =  \left| \Fourier \kf^{\left( \langle a+s \rangle \right)}(\tau) - \Fourier \kf^{\left( \langle a+s \rangle \right)}
(0) \right|
 \leq  L_{\kf} |hu|^{a+s - \langle a+s \rangle }.
\end{eqnarray}
From (\ref{Bias_Sob-Sinc}),  (\ref{Abschaetzung_Hoelderkern_1}) and (\ref{Abschaetzung_Hoelderkern_2}) we derive 
\begin{eqnarray}
 \notag &\phantom{\leq} & \sup_{u\in \R} \left|1-\Fourier \kf (hu)\right|^2 (1+|u|^2)^{-a-s} \leq  \left(\frac{L_K}{\langle a+s \rangle !}  +2/\pi^{a+s} \right)^2 h^{2(a+s)  }.
\end{eqnarray}
 This gives the desired result thanks to (\ref{Darstellung_Bias_Sobolev}) and (\ref{Darstellung_Bias_Sobolev_2.5.}).
\end{proof}

%%%%%%%%%%%%%%%%%%%%%%%%%%%%%%%%%%%
%%   Lokale Konvergenzraten
%%%%%%%%%%%%%%%%%%%%%%%
\begin{lemma}
 \label{Lemma_Abschaetzung_Bias_allgemeiner_Kernschaetzer}
In the situation of Theorem \ref{Hauptsatz_Raten_Hoelder}, 
\begin{equation*}
 \label{Formel_Biasabschaetzung}
\left| \int f(x)  \mu(\d x) - \int f(x) (\kf_h \ast \mu)(x) \d x \right|^2  \leq C_{B} h^{2a+2s}
\end{equation*}
for some positive constant $C_B$ depending on $C_K, a-d_1\vee d_2 -b$, $L$ and $R$.
\end{lemma}
\begin{proof}
We use the trivial observation that the local density $g_D$ can be extended to a compactly supported function  $g_1\in \mathcal{H}_{\R}(a,L',R')$
with constants $L'>L$ and $R'>R$.  

Let $\mu_1$ be the signed measure with density $g_1$ and let $\mu_2:= \mu- \mu_1$. Then the approximation error can be decomposed as follows: Since
$\supp(f)=[a,b]\subseteq D$, we have
\begin{align}
\notag \phantom{=} & \,   \left| \int f(x) \mu(\d x) - \int f(x) (\kf_h\ast \mu)(x) \d x \right| \\
\notag  = & \left|\int f(x) (g_1(x) - \kf_h\ast g_1(x) ) \d x-\int f(x) (\kf_h\ast \mu_2)(x)\d x  \right|\\
\label{letzte_Zeile}\leq &\,   \left|\int f(x) (g_1- \kf_h\ast g_1)(x) \d x  \right|
+ \left| \int f(x) (\kf_h\ast \mu_2)(x) \d x \right| .
\end{align}
Recall that $f$ is compactly supported and, by assumption, \mbox{$|\Fourier f(u)|\leq C_f (1+|u|)^{-s}$}.  Consequently,  $f$ is a distribution of order $k:=-s$, see  Theorem 7.23 in \citet{Rudin}.   Thanks to Theorem 6.34 ibidem, we can estimate 
\begin{eqnarray*}
 \label{Bias_allgKS_Rest}
 \left| \int f(x) \left( \kf_h\ast \mu_2\right)(x) \d x \right| 
  \leq  \|f\|\sup_{m\leq k} \sup_{x\in [a,b]}\limits \left|(\kf_h\ast \mu_2)^{(m)}(x)  \right|
\end{eqnarray*}
for some constant $\|f\|$ depending only on  $f$. 
By assumption,  $\kf_h$ is $k$-times continuously differentiable, with bounded derivatives. Since, moreover, $\mu_2$ is finite, we can derivate under the integral sign and write
\begin{eqnarray}
\label{Bias_allgKS_Rest_2.5.}
\notag &\phantom{=}& \|f\|\sup_{m\leq k} \sup_{x\in [a,b]}\limits \left|(\kf_h\ast \mu_2)^{(m)}(x)  \right|= \|f\| \sup_{m\leq k}\sup_{x\in [a,b]}\limits \left| \left( \kf_h^{(m)}\ast \mu_2 \right)(x) \right| \\
 & = & \|f\| \sup_{m\leq k} \sup_{x\in [a,b]}\limits \left| \int  h^{-m-1} \kf^{(m)}\left(\frac{x-y}{h} \right)\mu_2(\d y)
 \right|.
\end{eqnarray}
Using the fact that $\supp(f)= [a,b]  \subseteq D= (d_1, d_2)$ and that $\mu_2\big|_D\equiv 0$ by construction, we continue from
(\ref{Bias_allgKS_Rest_2.5.})  by estimating
\begin{eqnarray*}
\notag    \|f\| \sup_{m\leq k} \sup_{x\in [a,b]}\limits \left| \int  h^{-m-1} \kf^{(m)}\left(\frac{x-y}{h} \right)\mu_2(\d y)
\right| \leq  \|f\| \sup_{m\leq k} \sup_{z\geq \delta }\limits
h^{-m-1} \left| \kf^{(m)}\left(\frac{z}{h}\right)\right| |\mu_2|(\R) 
\end{eqnarray*}
with $\delta:= (a-d_1) \wedge (d_2-b)$.
Finally, the assumptions on the decay of $\kf$ and its derivatives up to order $k$ give
\begin{equation*}
 \|f\| \sup_{m\leq k} \sup_{z\geq \delta }\limits
h^{-m-1} \left| \kf^{(m)}\left(\frac{z}{h}\right)\right| |\mu_2|(\R)
 \leq  \|f\| C_K |\mu_2|(\R)  \delta^{-a-(s\vee 0)-1} h^{a+s}.
\end{equation*}
It remains to consider the first expression in the last line of formula (\ref{letzte_Zeile}). We observe that 
\begin{eqnarray}
\label{Distribution_aso}
\int f(x) \left(\kf_h\ast g_1\right)(x) \d x = \int \kf_h(y) \int f(x) g_1(x-y) \d x \d y.
\end{eqnarray}
If $f$ is regular, this is simply a consequence of Fubini\rq{}s theorem. For non-regular $f$, a straightforward  generalization of Theorem 39.3  and Theorem 39.10 in \citet{Jantscher} gives formula (\ref{Distribution_aso}). 

Let  $\tilde{g}_1(x):= g_1(-x)$. Then  
\begin{eqnarray}
\notag \int f(x) (g_1(x) - \kf_h\ast g_1(x) ) \d x &{=}& \int \kf_h(y) \left( \int g_1(x) f(x) \d x - \int g_1(x-y) f(x) \d x \right) \d y  \\
\notag & =& \int \kf_h(y) \left( f\ast \tilde{g}_1 (0) - f\ast \tilde{g}_1(y) \right) \d y.
\end{eqnarray}
We observe  that $f\ast \tilde{g}_1$ is $\langle a \rangle +s$-times continuously differentiable and that the derivative
of order $\langle a \rangle +s$ is $a-\langle a\rangle$-H\"older continuous. To see this, we use the fact that $\left(f\ast \tilde{g}_1\right)^{\left(\langle a\rangle +s \right)} = f^{(s)} \ast \tilde{g}_1^{\langle a \rangle }$, where 
$f^{(s)}$ is understood to be  a distributional derivative.  (See  Theorem  41.3 in \cite{Jantscher} for explanation.) Now, since
\mbox{$|\Fourier f^{(s)}(-u)| = |u|^s |\Fourier f(-u)| \leq C_f$}, we can use  Theorem 7.23 in \cite{Rudin} to conclude that $f^{(s)}$ 
is a compactly supported distribution of order $0$.  From this, and Theorem 6.34 in \cite{Jantscher},  we derive that 
\begin{eqnarray*}
\notag  \left| f^{(s)}\ast \tilde{g}_1^{\langle a \rangle } (x) - f^{(s)} \ast \tilde{g}_1^{\langle a \rangle}(y) \right|   & \leq &  \|f^{(s)}\|  \sup_{\tau \in [a,b]} \left| \tilde{g}_1^{\langle a \rangle } (x-\tau) - \tilde{g}_1^{\langle a \rangle } (y-\tau) \right|\\
&\leq & \|f^{(s)} \| L' \left| x- y\right|^{a- \langle a\rangle }
\end{eqnarray*}
for a constant $\|f^{(s)}\|$ depending on $f$.

Now, a  Taylor series expansion  of $f\ast \tilde{g}_1$ and an application of the order of $\kf$ yields for some
$\tau_y \in [0,y]$:
\begin{eqnarray*}
\notag &\phantom{\leq} & \left|  \int \kf_h(y)  \left( (f\ast \tilde{g}_1)(0) - (f\ast \tilde{g}_1)(y)     \right)  \d y  \right|  \\
\notag & = & \frac{1}{\left(\langle a\rangle +s\right)!}\hspace*{-0,1cm} \left|\int \kf_h(y)\hspace*{-0,1cm} \left((f\ast \tilde{g}_1 )^{(\langle a\rangle +s)}(\tau_y)
-(f\ast \tilde{g}_1 )^{(\langle a\rangle +s)}(0) \right) y^{\langle a \rangle +s} \d y  \right|\\
\notag  &\leq &   \frac{1}{\left(\langle a\rangle +s\right)!} \|f^{(s)}\| L'\int |\kf_h(y)| |y|^{a-\langle a \rangle} |y|^{\langle a \rangle +s} \d y\\
&= & \frac{1}{\left(\langle a\rangle +s\right)!} \|f^{(s)}\| L'  \int |\kf(z)| |z|^{a+s} \d z\, h^{a+s}. 
\end{eqnarray*}
This completes the proof.
\end{proof}
Theorem \ref{Hauptsatz_Konvergenzraten_Sobolev} and Theorem \ref{Hauptsatz_Raten_Hoelder} are immediate consequences of Lemma \ref{Lemma_Abschaetzung_Bias_allgemeiner_Kernschaetzer} and Lemma \ref{Lemma_Abschaetzung_Bias_allgemeiner_Kernschaetzer} combined with the assumptions on $\kf$ and $f$ and on the decay of $\phi$. 

\subsection{Adaptive estimation }
\label{Beweisteil_adaptiv}
\subsubsection{{Preliminaries}}  Recall the following well known result: 
\begin{lemma}[Talagrand's inequality]
\label{Talagrand-Ungleichung}
Let $I$ be some countable index set. For each $i\in I$, let $X_1^{(i)}, \cdots, X_n^{(i)}$ be centred i.i.d.
complex valued random variables, defined on the same probability space,  with
$\|X_1^{(i)}\|_{\infty} \leq B$ for some $B<\infty$. Let $v^2:= \sup_{i\in I}\limits \Var X_1^{(i)}$. Then for arbitrary $\eps>0$,
 there are positive  constants $c_1$ and  $c_2=c_2(\eps)$ depending only on $\eps$ such that for any $\kappa>0$:
\begin{equation*}
 \PP \left(\left\{ \sup_{i\in I} |S_n^{(i)}| \geq (1+\eps) \E \left[\sup_{i\in I} |S_n^{(i)}| \right] +\kappa \right\} \right)
\leq 2  \exp\left(-n \left(  \frac{\kappa^2}{c_1 v^2} \wedge \frac{\kappa}{ c_2  B} \right) \right).
\end{equation*}
\end{lemma} 
A proof can be found for example in \citet{Massart}.

The main objective of the present subsection is to prove  Lemma \ref{Neumann_Lemma_uniform}.  We start by providing a series of technical results.   The arguments presented here are fairly general and not particular to the \Levy\ model. 
\begin{lemma}
\label{Lemma_Abweichung_emp_cf_uniform} Let $\tau>0$ be given. Let $\delta$ be the constant appearing in the definition of the
weight function $w$  and let $c_1$ be  the constant in Talagrand's inequality. Then, for arbitrary $\gamma>0$, there is a positive constant
\mbox{$C\kap=C\kap^{\tau,\gamma, \delta}$} depending only on the choice of $\tau, \gamma$ and $\delta$  such that  we have for $n\geq 1$:
\begin{equation*}
\PP \left(\left\{ \exists u\in \R: |\hat{\phi}_n(u) - \phi(u)|\geq  \tau (\log n)^{1/2} w(u)^{-1} n^{-1/2}  \right\} \right)
\leq C\kap  n^{-\frac{(\tau - \gamma)^2}{c_1} }.
\end{equation*}
\end{lemma}
\begin{proof}
 We prove the claim for the countable set of rational numbers. By continuity of the characteristic function and of
$w$, it carries over to the whole range of real numbers.

By Theorem 4.1 in \citet*{Neumann_Reiss}, we have for some positive  constant $C\neurei$:
\begin{equation*}
 \E \left[\sup_{u\in \R} |\hat{\phi}_n(u) - \phi(u)|w(u) \right]  \leq C\neurei n^{-1/2}.
\end{equation*}
Since moreover, we trivially have $
\sup_{u\in \R}  \Var[\hat{\phi}_1(u)] \leq 1$ and 
$ \sup_{u\in \R} \|\hat{\phi}_1(u)w(u)\|_{\infty}\leq 1,$
we can apply Talagrand's inequality. Setting
\begin{equation*}
 \kappa_n:=  \tau (\log n)^{1/2}  n^{-1/2} - (1+\eps) C\neurei  n^{-1/2},
\end{equation*}
for some $\eps>0$,   we    can   estimate
\begin{eqnarray*}
\notag &\phantom{=}& \PP \left( \left\{\exists q \in \Q: |\hat{\phi}_n(q) - \phi(q)|\geq  \tau (\log n)^{1/2} w(q)^{-1} n^{-1/2}
 \right\}\right)   \\
\notag &=& \PP \left(\left\{ \sup_{q\in \Q} |\hat{\phi}_n(q) - \phi(q)|w(q)\geq  \tau (\log n)^{1/2}
n^{-1/2}   \right\}
\right)  \\
\notag &\leq & \PP \left( \left\{ \sup_{q\in \Q} |\hat{\phi}_n(q) - \phi(q)|w(q)\geq  (1+\eps)
\E \left[ \sup_{q\in \Q} |\hat{\phi}_n(q) - \phi(q)|w(q)  \right]    +  \kappa_n \right\}  \right)\\
&\leq & 2 \exp \left(-n \left(  \frac{\kappa_n^2}{c_1} \wedge  \frac{\kappa_n}{c_2}\right) \right).
\end{eqnarray*}
By definition of $\kappa_n$,  we have for $C\kap$ large enough and  arbitrary $n\geq 1$:
\begin{eqnarray*}
\notag &&\hspace*{-0,6cm}\phantom{\leq}  2\exp \left(-n \left( \frac{\kappa_n^2}{c_1} \wedge \frac{\kappa_n}{c_2} \right) \right) \\
\notag  &&\hspace*{-0,6cm}\leq    2 \exp \left(  -  \frac{\left( \tau (\log n)^{1/2} - (1+\eps ) C\neurei\right)^2 }
{c_1} \right) \vee 2 \exp \left( -\frac{n^{1/2} \left(\tau  (\log n)^{1/2} - (1+\eps) C\neurei    \right)}{c_2  }  \right) \\
&&\hspace*{-0,6cm} \leq  C\kap  \exp \left( - \frac{(\tau -\gamma)^2}{c_1} (\log n)  \right) = C\kap n^{-\frac{(\tau - \gamma)^2}{c_1}   }.
\end{eqnarray*}
This is the desired result for the rational numbers  and hence, by continuity, for
the real line.
\end{proof}
We can now use Lemma \ref{Lemma_Abweichung_emp_cf_uniform} to analyse the deviation of
$1/\check{\phi}_n$ from  $1/\phi$.
\begin{lemma}
\label{Abweichung_Bruch_emp_cf_uniform}
Assume that for some $\gamma>0$ and some $p>0$, we have $\kappa\geq 2(\sqrt{p c_1} + \gamma)$,  where $c_1$
 denotes the constant in Talagrand's inequality. Then 
\begin{equation*}
\PP \left(\left\{\exists u \in \R:   \left|\phinneuu  - \phinu  \right|^2>\left(\frac{(4\kappa)^2(\log n) w(u)^{-1} n^{-1}}{|\phi(u)|^4} \wedge 
   \frac{(5/2)^2}{|\phi(u)|^2} \right)\right\}  \right)\leq  C\kap n^{-p} .
\end{equation*}
\end{lemma}
\begin{proof}
We use the notation $a_n(u): =\kappa  (\log n)^{1/2}  w(u)^{-1} n^{-1/2 } $.  Let us introduce the favourable set 
\[
C:= \Big\{ \forall u\in \R:  |\hat{\phi}_n(u) - \phi(u)|  \leq a_n(u)/2 \Big\}.
\]
Thanks to the preceding Lemma and the choice of $\kappa$, 
\[
\PP(C^c) \leq C\kap n^{-p},
\]
so it is enough to consider $C$.  

We have 
\[
\left| \frac{1}{\check{\phi}_n(u) }- \frac{1}{\phi(u)}         \right|^2= \frac{ |\check{\phi}_n(u)  - \phi(u) |^2    }{ |\check{\phi}_n(u)|^2| \phi(u)|^2  }.
\]
Consider first the case where $|\phi(u)| > 3 a_n(u)/2$.  Then, by definition of $C$, 
\[
|\hat{\phi}_n(u) | \geq |\phi(u)|- |\phi(u)- \hat{\phi}_n(u)| > \frac{2}{3} |\phi(u)|\geq a_n(u). 
\]
Consequently, 
\begin{align*} \frac{ |\check{\phi}_n(u)  - \phi(u) |^2    }{ |\check{\phi}_n(u)|^2| \phi(u)|^2  }  = \frac{ |\hat{\phi}_n(u)  - \phi(u) |^2    }{ |\hat{\phi}_n(u)|^2| \phi(u)|^2  }  \leq \frac{(3a_n(u)/4)^2    }{|\phi(u)|^4   }  \leq \frac{1}{4|\phi(u)|^2}. 
\end{align*}
Now, consider $|\phi(u)| \leq 3 a_n(u)/2$.  By definition of $C$, 
\begin{align*}
|\check{\phi}_n(u) - \phi(u)| \leq \max \{ |\hat{\phi}_n(u) - \phi(u)|, |\phi(u)| +a_n(u) \} \leq 5 a_n(u) /2. 
\end{align*}
Since, moreover,  $ \left|\check{\phi}_n(u) \right|\geq a_n(u)$ holds by definition of $\check{\phi}_n$, 
\begin{align*}
 \frac{ |\check{\phi}_n(u)  - \phi(u) |^2    }{ |\check{\phi}_n(u)|^2| \phi(u)|^2  }
\leq  \frac{  (5 a_n(u)/2)^2 }{|\check{\phi}_n(u)|^2| \phi(u)|^2 }  \leq \frac{ (5/2)^2}{|\phi(u)|^2 }
\leq \frac{ (15 a_n(u) /4)^2 }{|\phi(u)|^4}.
\end{align*}
This is the desired result. 
\end{proof}
The following useful corollary is an immediate consequence of the preceding statement. 
\begin{corollary}
\label{Korollar_Abweichung_Bruch_emp_cf_uniform}  In the situation of the preceding Lemma, 
\begin{equation*}
\label{Abweichung_Bruch_uniform_empirisch}
\PP \left( \left\{\exists u \in \R:   \left| \phinneuu - \phinu \right|^2 >
\left( \frac{5}{2} \kappa\right)^2 \frac{ (\log n) w(u)^{-2} n^{-1}    }{|{\check{\phi}}_n(u)|^2 |\phi(u)|^2}    \right\}  \right) \leq C\kap n^{-p}. 
\end{equation*}
\end{corollary}
The uniform version of the classical Neumann Lemma can now be stated as an easy consequence of Lemma \ref{Abweichung_Bruch_emp_cf_uniform}.
\begin{proof}[\textbf{Proof of Lemma \ref{Neumann_Lemma_uniform} }]
We use the notation 
\[
A_n(u):= \frac{  \left|\phinneuu -\phinu  \right|^2  }{ \frac{(\log n) w(u)^{-2} n^{-1}}{|\phi(u)|^4 }  \wedge \frac{1}{|\phi(u)|^2}}   
\]

Let the set $C$ be defined as in the proof of Lemma  \ref{Abweichung_Bruch_emp_cf_uniform}.
We can decompose 
\begin{align}
\notag \phantom{\leq} &\,\,  \E \left[\sup_{u\in \R} A_n(u)   \right]  =\E \left[ \sup_{u\in \R}A_n(u) 1_C  \right]   + \E \left[\sup_{u\in \R} A_n(u) 1_{C^c}   \right]
\end{align}
The definition of $C$, together with Proposition \ref{Abweichung_Bruch_emp_cf_uniform}, readily implies 
\begin{equation*}
 \E \left[\sup_{u\in \R} A_n(u) 1_C \right] \leq  \E \left[\sup_{u\in \R}  \frac{    \left( 4\kappa \right)^2 \frac{(\log n) w(u)^{-1} n^{-1}}{|\phi(u)|^4} \wedge \left(\frac{5}{2}\right)^2
   \frac{1}{|\phi(u)|^2}              }{ \frac{(\log n) w(u)^{-2} n^{-1}}{|\phi(u)|^4 }  \wedge \frac{1}{|\phi(u)|^2}} 1_C                               \right]
 \leq 16\kappa^2.
\end{equation*}
On the other hand, since we have, by definition $ 1/\check{\phi}_n(u) \leq a_n(u)^{-1}$, we can estimate for arbitrary $u\in \R$: 
\begin{equation*}
 \label{NLU_4} A_n(u)
\leq  |\phi(u)|^4 
 \frac{\left(a_n(u)^{-1}+\frac{1}{|\phi(u)|} \right)^2 }{{ a_n(u)^2 }\wedge |\phi(u)|^2   } \leq (2 a_n(u)^{-2}  +1 )^2 
\leq   (2 \kappa^{-2} +1)^2 n^2,                    
\end{equation*}
so
\begin{eqnarray*}
\label{NLU_5}    \E \left[\sup_{u\in \R}A_n(u)  1_{C^c }   \right]
\leq (2\kappa^{-2}+1)^2 n^2 \PP \left( C^c  \right).
\end{eqnarray*}
Since  Lemma \ref{Lemma_Abweichung_emp_cf_uniform} implies
$  \PP( C^c) \leq C\kap n^{-2}$,  this gives the desired result. 
\end{proof}
The result immediately extends to values different from $2$. The following Corollary can be obtained, replacing in each step $2$ by $2q$:  
\begin{corollary}\label{Korollar_Neumann_uniform} In the situation of the preceding statement, let $\kappa\geq 2\left( \sqrt{2qc_1}+\gamma\right)$ for some $q\in \R^+$.  Then we have for 
some constant $C\neukap=C\neukap_q^{\gamma,\delta,\eps}$ depending on $\gamma$, $q, \delta$ and $\eps$:
 \begin{eqnarray}
\E \left[ \sup_{u\in \R}\frac{  \left|\phinneuu -\phinu  \right|^{2q}  }{ \frac{(\log n)^q w(u)^{-2q} n^{-q}}{|\phi(u)|^{4q} }
\wedge \frac{1}{|\phi(u)|^{2q}}}    \right]
\leq C\neukap.
\end{eqnarray}
\end{corollary}
\subsubsection{Auxiliary results}
The main result of the present subsection is Proposition \ref{wichtigstes_Hilfsresultat}, which is then the most important technical tool for analysing the adaptive bandwidth selection. 

We use, in the prequel, the short notation 
\[
\Delta_{m,k}\Fourier \kf(u):= \Kerndifmk.
\]
\renewcommand{\Kerndifmk}{\Delta_{m,k}\Fourier \kf(u)}
\begin{lemma} \label{Lemma_Abweichung_Erweiterung} For $k,m\in \N$,  let
\begin{equation*}
 x_{{f_{m,k}}}^2:= \frac{1}{2\pi^2} \Bigg\{C_1 \int  \left|\Fourier  f(\text{-}u)\Kerndifmk\right|^2 \d u    \wedge C_2
  \left(  \int  \left|\Fourier  f(\text{-}u)\Kerndifmk \right| \d u    \right)^2 \Bigg\}.
\end{equation*}
Moreover,  let
\begin{equation*}
\lambda_{{f_{m,k}}}:= \sqrt{c_1} \log \left( \xfmk^2  (k-m)^2  \right).
\end{equation*}
For some $\gamma>0$, let $\kappa= 2( \sqrt{2 p c_1} +\gamma)$.  Then we have for some constant $C\kap$ depending on $\gamma,\delta$ and $\eps$:
\begin{eqnarray*}
&\phantom{\leq}& \PP \left(  \left\{\exists u \in \R:  |\hat{\phi}_n(u) - \phi(u)|
 \geq \left( \frac{\kappa}{2} (\log n)^{1/2}  +\lambda_{{f_{m,k}}}\right) w(u)^{-1}  n^{-1/2} \right\}  \right) \\
& \leq &  C\kap n^{-p}   \xfmk^{-2} (k-m)^{-2}.
\end{eqnarray*}
\begin{proof}
The proof runs  along the same lines as the proof of Lemma \ref{Lemma_Abweichung_emp_cf_uniform}, setting, this time
\[
\kappa_n:= \left( \frac{\kappa}{2} (\log n)^{1/2}   + \lambda_{{f_{m,k}}} \right) n^{-1/2}  - C\neurei n^{-1/2}.
\]
Using again continuity of the (empirical) characteristic function,  the Talagrand inequality and the choice of $\kappa$, we derive that for $C\kap$ chosen large enough,
\begin{eqnarray*}
\notag & \phantom{\leq} & \PP \left(  \left\{\exists u \in \R:  |\hat{\phi}_n(u) - \phi(u)|
 \geq \left( \frac{\kappa}{2} (\log n)^{1/2}  +\lambda_{{f_{m,k}}}\right) w(u)^{-1}  n^{-1/2} \right\}  \right)  \\
\notag &\leq & 2 \exp \left(-\frac{\left(\frac{\kappa}{2} (\log n)^{1/2} +\lambda_{{f_{m,k}} } - C\neurei \right)^2}{c_1} \right)  \vee\,  2 \exp \left(-\frac{n^{1/2}\left(\frac{\kappa}{2}(\log n)^{1/2} +\lambda_{{f_{m,k}}}-C\neurei   \right) }{c_2}       \right)\\
\notag &\leq & C\kap  \exp\left(- \frac{\left(  \kappa/2 -\gamma \right)^2}{c_1} (\log n) - \log \left( x_{{f_{m,k} }}^2 (k-m)^2 \right)   \right) \\
& = & C\kap n^{-p}  x_{{f_{m,k}}}^{-2} (k-m)^{-2}.
\end{eqnarray*}

\end{proof}
\end{lemma}
The above result implies the following extension of Corollary  \ref{Abweichung_Bruch_emp_cf_uniform}:
\begin{corollary}
\label{Lemma_Abweichung_emp_cf_uniform_Erweiterung}In the situation of the preceding statement, we have for some constant $C\kap$ depending on $\gamma$ and $\delta$:
\begin{eqnarray*}
&\phantom{\leq} & \PP \left( \left\{\exists u \in \R:  \left|\phinneuu - \phinu  \right|^2 >
\frac{ \left(\frac{5}{2}  \kappa   (\log n)  + \lambda_{{f_{m,k}}} \right)^2 w(u)^{-2} n^{-1}}{|\check{{\phi}}_n(u)|^2|\phi(u)|^2 }  \right\} \right) \\
&\leq & C\kap   n^{-p} \xfmk^{-2} (k-m)^{-2}.
\end{eqnarray*}
\end{corollary}
\begin{comment}
\begin{proof}
This statement is derived from Lemma \ref{Lemma_Abweichung_Erweiterung} in the same way as Corollary  \ref{Korollar_Abweichung_Bruch_emp_cf_uniform}
is derived from Lemma \ref{Lemma_Abweichung_emp_cf_uniform} and from the proof of Lemma \ref{Abweichung_Bruch_emp_cf_uniform}.
We only have to replace, in each step,   $\frac{\kappa}{2} (\log n)^{1/2} w(u)^{-1}$ by $\left(\frac{\kappa}{2} (\log n)^{1/2} +\lambda_{m,k}\right) w(u)^{-1}$.
\end{proof}
\end{comment}
\begin{proposition}
\label{wichtigstes_Hilfsresultat} Assume that the conditions which are summarized in Theorem
\ref{Hauptsatz_adaptiver_Schaetzer_Levy} are satisfied.
 Then we can estimate for arbitrary $m\in \mathbb{N}$:
\begin{eqnarray*}
\label{Ungleichung_Levy_wichtigstes_Hilfsresultat}
\E \left[  {\sup_{k>m,k\in \N}}  \left\{ \left| (\thetaneu_k-\thetaneu_m) -(\theta_{k} -
\theta_{m} )  \right|^2  -\frac{1}{2}\widetilde{\strafH}(m,k)^2\right\}_+  \right]\leq C n^{-1},
\end{eqnarray*}
where $C$ is a positive constant depending on the exponential moment. 
 \end{proposition}
\begin{proof}Let 
\begin{equation*}
 \tilde{\theta}_m:= \frac{1}{2\pi}  \int \Fourier f(-u) \frac{\phi'(u)}{\check{\phi}_n(u) } \Fourier \kf
\left(u/m\right) \d u.
\end{equation*}
We use the estimate
 \begin{eqnarray}
 \notag & \phantom{\leq}&    \E \left[
{\sup_{k>m, k\in \N}} \left\{ \left|(\hat{{\theta}}_k- \hat{{\theta}}_{m}) - (\theta_{k} - \theta_{m})   \right|^2 -
\frac{1}{2} \tilde{\strafH}^2(m,k)\right\}_+   \right]  \\
 \notag  &\leq & 2 \E \left[{ \sup_{k>m,k\in\N} } \left\{ \left|(\hat{{\theta}}_k- \hat{{\theta}}_{m}) -
(\tilde{\theta}_{k} - \tilde{\theta}_{m})   \right|^2 -
\frac{1}{8}\tilde{\strafH}^2(m,k)\right\}_+  \right] \\
\label{E_sup_Zeile_2}&+&   2 \E \left[{\sup_{k >  m, k\in \N} } \left\{ \left|(\tilde{\theta}_k- \tilde{\theta}_{m}) -
(\theta_{k} - \theta_{m})   \right|^2 - \frac{1}{8} \tilde{\strafH}^2(m,k)   \right\}_+  \right].
 \end{eqnarray}
 Consider first the expression appearing in the second line of formula  (\ref{E_sup_Zeile_2}). We can estimate, conditioning on $\hat{\phi}_n$:
 \begin{eqnarray*}
\notag &\phantom{\leq} &  \E \left[{ \sup_{k>m,k\in\N} }  \left\{ \left|(\hat{{\theta}}_k- \hat{{\theta}}_{m}) -
(\tilde{\theta}_{k} - \tilde{\theta}_{m})   \right|^2 - \frac{1}{8} \tilde{\strafH}^2(m,k)   \right\}_+  \right]  \\
\notag & = & \E\left[ \E \left[{\sup_{k>m,k\in\N}}\left\{ \left|(\hat{{\theta}}_k- \hat{{\theta}}_{m}) -
(\tilde{\theta}_{k} - \tilde{\theta}_{m})   \right|^2 - \frac{1}{8} \tilde{\strafH}^2(m,k)   \right\}_+ \Bigg|\hat{\phi}_n\right]  \right] \\
&\leq & \E \left[ {\sum_{k>m,k\in\N}\limits }\E \left[  \left\{ \left|(\hat{{\theta}}_k- \hat{{\theta}}_{m}) -
(\tilde{\theta}_{k} - \tilde{\theta}_{m})   \right|^2 - \frac{1}{8} \tilde{\strafH}^2(m,k)   \right\}_+ \Bigg| \hat{\phi}_n  \right]
 \right].\phantom{mllllll}
\end{eqnarray*}
Unlike in a classical density deconvolution model, where characteristic functions are being considered, we have to deal, in the present situation,  with the additional complication that $\hat{\phi}_n$ is unbounded. To be able to apply the  Bernstein inequality, let us introduce the truncated version of $Z_j$, 
\begin{equation*}
\label{Abgeschnitten}
\bar{Z}_j:= Z_j 1_{ \left\{|Z_j|\leq    \frac{4}{\eta}   \left(
 \log n + \log \tilde{x}_{m,k } (k-m)  \right) \right\} }.
 \end{equation*}
Moreover, we define the remainder term
$Z_j^r := Z_j - \bar{Z}_j$. Then 
\renewcommand{\Kerndifmk}{\Delta_{m,k} \Fourier \kf(u) }
\begin{eqnarray*}
\notag \label{Abschneiden_1}&\phantom{\leq}& \left|(\hat{{\theta}}_k - \hat{{\theta}}_m )- (\tilde{\theta}_k - \tilde{\theta}_m)  \right|^2  \\
\notag \label{Abschneiden_2}& = & \left|\frac{1}{n} \sum_{j=1}^n\limits \frac{1}{2\pi}  \int \Fourier f(-u)
\frac{ (Z_j e^{iuZ_j} - \E [Z_1 e^{iuZ_1} ] )}{{\check{\phi}}_n(u) }  \Kerndifmk  \d u  \right|^2 \\
\notag \label{Abschneiden_3} &\leq & 2  \left|\frac{1}{2\pi} \frac{1}{n} \sum_{j=1}^n \int \Fourier f(-u)
 \frac{\bar{Z}_j e^{iuZ_j} - \E [\bar{Z}_1 e^{iuZ_1}]}{{\check{\phi}}_n(u)}  \Kerndifmk\d u     \right|^2   \\
\notag \label{Abschneiden_4}&+&  2 \left| \frac{1}{2\pi} \frac{1}{n} \sum_{j=1}^n \int \Fourier f(-u)
 \frac{Z^r_j e^{iuZ_j^r} - \E [Z^r_1 e^{iuZ^r_1}]}{{\check{\phi}}_n(u)} \Kerndifmk \d u  \right|^2 \\
\label{Abschneiden_5}&=:&2 \left| \left( \hat{\bar{\theta}}_{k} -\hat{\bar{\theta}}_m  \right)
-\left(  \tilde{\bar{\theta}}_k -\tilde{\bar{\theta}}_m \right) \right|^2
+ 2 \left| \left( \hat{\theta}_k^r - \hat{\theta}_m^r \right) - \left(  \tilde{\theta}^r_k - \tilde{\theta}^r_m \right) \right|^2.
\end{eqnarray*}
Since
\begin{eqnarray*}
 \E \left[  \left| \left( \hat{\bar{\theta}}_{k} -\hat{\bar{\theta}}_m  \right)
-\left(  \tilde{\bar{\theta}}_k -\tilde{\bar{\theta}}_m \right) \right|^2  \Big| \hat{\phi}_n \right]
 \leq \frac{1}{n}\tilde{\sigma}_{m,k}^2  \ a.s.
\end{eqnarray*}
and 
\begin{eqnarray}
\notag &\phantom{\leq} & \left\| \frac{1}{2\pi}\int  \Fourier f(-u)
\frac{\bar{Z}_1 e^{iuZ_1} }{{\check{\phi}}_n(u)}    \Kerndifmk  \d u  \right\|_{\infty} \\
\notag & \leq &  \frac{4}{\eta} (\log (n \tilde{x}_{m,k} (k-m)))\frac{1}{2\pi} \int  \left| \frac{\Fourier f(\text{-}u)}{{\check{\phi}}_n(u)}\right| \left|\Kerndifmk   \right|   \d u  \\
\notag &\leq &    \frac{4\sqrt{n}}{\eta} (\log (n  \tilde{x}_{m,k} (k-m) ))  \tilde{x}_{m,k}  \   \    \  a.s,
\end{eqnarray}
the integral version of the classical Bernstein-inequality (see, for example, \citet{Dudley}) yields 
\begin{eqnarray}
\notag &\phantom{\leq} & \E \left[\left\{ \left| \left( \hat{\bar{\theta}}_{k} -\hat{\bar{\theta}}_m  \right)
-\left(  \tilde{\bar{\theta}}_k -\tilde{\bar{\theta}}_m \right) \right|^2 -\frac{1}{8}\tilde{\strafH}^2(m,k)\right\}_+ \Big|\hat{\phi}_n    \right]  \\
\notag  &\leq &     \frac{2048\sqrt{2}}{\eta^2} \frac{ (\log \left(n  \tilde{x}_{m,k}(k-m)\right) )^2 \tilde{x}_{m,k}^2 }{n}   \exp \left(-\frac{c^{\pen}  \tilde{\lambda}_{m,k}  \tilde{x}_{m,k}}
 { \frac{256}{\eta}\left(\log \left(n \tilde{x}_{m,k}(k-m)\right)\right) \tilde{x}_{m,k} } \right)\\
\notag &\phantom{=}&+ 32 \frac{\tilde{\sigma}_{m,k}^2}{n}
 \exp\left(- \frac{ c^{\pen} \tilde{\lambda}_{m,k}^2  \tilde{\sigma}_{m,k}^2}{ 64 \tilde{\sigma}_{m,k}^2  }  \right)  \  \  a.s.
\end{eqnarray}
It is important to recall at this point  
that $\hat{\phi}'_n$ and $\hat{\phi}_n$ are independent by construction.

Using  the fact that, by definition,  $c^{\pen}\geq 64$ and
\mbox{$\tilde{\lambda}_{m,k}^2 \geq \log \left(\tilde{\sigma}_{m,k}^2 (k-m)^2 \right)$}, as well as
\[
 \tilde{\lambda}_{m,k} \geq  \frac{8}{\eta} \left(\log \left(n  \tilde{x}_{m,k} (k-m) \right)\right)
\log \left( \log \left(n  \tilde{x}_{m,k} (k-m) \right)\right)^2
\log \left(\tilde{x}_{m,k}^2 (k-m)^2 \right),
\] 
we can continue by estimating
\begin{align}
\notag \phantom{\leq}& \,  \,32 \frac{\tilde{\sigma}_{m,k}^2}{n}
 \exp\left(- \frac{ c^{\pen} \tilde{\lambda}_{m,k}^2  \tilde{\sigma}_{m,k}^2}{ 64 \tilde{\sigma}_{m,k}^2  }  \right) \\
\notag &  + \frac{2048\sqrt{2}}{\eta^2} \frac{ (\log \left(n  \tilde{x}_{m,k}(k-m)\right) )^2 \tilde{x}_{m,k}^2 }{n}  \exp\hspace*{-0,07cm}
 \left(\hspace*{-0,07cm}-\frac{c^{\pen}  \tilde{\lambda}_{m,k}  \tilde{x}_{m,k}}
 {\frac{256}{\eta}\left(\log \left(n \tilde{x}_{m,k}(k-m)\right)\right) \tilde{x}_{m,k} } \right) \\
\notag \leq& \,  \, 32 \frac{\tilde{\sigma}_{m,k}^2}{n}
\exp\left(-  \tilde{\lambda}_{m,k}^2    \right) \\
\notag &+ \frac{2048\sqrt{2}}{\eta^2} \frac{ (\log \left(n  \tilde{x}_{m,k}(k-m)\right) )^2 \tilde{x}_{m,k}^2 }{n}
 \exp \left(- \frac{\tilde{\lambda}_{m,k} }
 {\frac{4}{\eta}  \log \left( n \tilde{x}_{m,k}(k-m)\right)  } \right) \\
\notag \leq& \,\,   32 \frac{\tilde{\sigma}_{m,k}^2}{n}  \tilde{\sigma}_{m,k}^{-2} (k-m)^{-2} \\
\notag  &+   \frac{2048\sqrt{2}}{\eta^2} \frac{ (\log \left(n  \tilde{x}_{m,k}(k-m)\right) )^2 \tilde{x}_{m,k}^2 }{n}
 (\log \left( n \tilde{x}_{m,k} (k-m)\right) )^{-2}  \tilde{x}_{m,k}^{-2}  (k- m)^{-2}.
\end{align}
We have thus shown that, almost surely,
\begin{eqnarray*}
\label{Abschaetzung_Abgeschnittenes_Levy}  \underset{k\in \mathbb{N}}{\sum_{k\geq m} }\E\left[ \left\{ \left| \left( \hat{\bar{\theta}}_{k} -\hat{\bar{\theta}}_m  \right)
-\left(  \tilde{\bar{\theta}}_k -\tilde{\bar{\theta}}_m \right) \right|^2 -\frac{1}{8}\tilde{\strafH}^2(m,k)      \right\}_+
\big| \hat{\phi}_n\right]
\leq  2 (64 +  \frac{2048\sqrt{2}}{\eta^2}  ) n^{-1} .
\end{eqnarray*}

The remainder term can be estimated as follows:
\begin{eqnarray}
\notag &\phantom{|}& \E \left[  \left| \left( \hat{\theta}_k^r- \hat{\theta}_m^r\right)-
\left( \tilde{\theta}_k^r- \tilde{\theta}_m^r\right)   \right|^2\Bigg|
\hat{\phi}_n \right]   \\
\notag &\leq&   \frac{1}{n}   \E \left[ \left| \frac{1}{2\pi} \int \Fourier f(\text{-}u)
 \frac{Z^r_1 e^{iuZ_1}}{{\check{\phi}}_n(u)}
 \Kerndifmk   \d u  \right|^2 \Bigg|\hat{\phi}_n \right]  \\
\notag   &  \leq  & \frac{1}{n} \frac{1}{(2\pi)^2}   \E \left[|Z_1^r|^2 \big|  \hat{\phi}_n \right]
  \left( \int \left|\frac{\Fourier f(\text{-}u)}{\check{{\phi}}_n(u)  }  \right|
   \left|\Kerndifmk  \right|\d u\right)^2 =  \tilde{x}_{m,k}^2   \E \left[ |Z_1^r|^2 \big| \hat{\phi}_n\right]\   \  a.s.
\end{eqnarray}
Now, we apply Markov's inequality to find 
\begin{eqnarray}
\notag  \E \left[ |Z_1^r|^2 \big| \hat{\phi}_n\right]&=&  \E \left[\left|Z_1\right|^2 1 \left(\{|Z_1| > \frac{4}{\eta}  \log\left( n  x_{m,k}(k-m)  \right)   \}\right)   \big| \hat{\phi}_n\right]  \\
\notag & \leq & \E \left[ |Z_1|^2 \exp \left( \frac{\eta}{2}|Z_1|\right) \bigg| \hat{\phi}_n\right]
\exp\left(-2 \log\left(n \tilde{x}_{m,k} (k-m)\right) \right)\\
\notag &\leq &  n^{-1} \frac{4}{\eta^2}\E \left[\exp\left(\eta |Z_1|\right)  \right]  \tilde{x}_{m,k}^{-2} (k-m)^{-2} \   \   a.s.
\end{eqnarray}
where we have used the notation $1(A)$ instead of $1_A$. 
We have thus shown 
\begin{equation*}
\label{Abschaetzung_Rest_Levy}
 \sum_{k\geq m,k\in \mathbb{N}} \E\left[\left| \left( \hat{\theta}_k^r- \hat{\theta}_m^r\right)
-\left( \tilde{\theta}_k^r-\tilde{\theta}_m^r \right)\right|^2   \Big|\hat{\phi}_n       \right] 
 \leq \frac{8}{\eta^2} \E \left[\exp \left( \eta|Z_1|\right) \right] n^{-1} \  \  
  a.s.
\end{equation*}
It remains to consider the expression appearing in the last line of formula  (\ref{E_sup_Zeile_2}).

Let us introduce, for arbitrary  $m\leq k$, the favourable set
\begin{equation*}
 C(m,k)   :=   \left\{ \forall u\in \R: \left|\phinneuu  -\phinu\right|^2 \leq
 \frac{   \left( \left(\frac{5}{2} \kappa\right)\hspace*{-0,08cm}(\log n)^{1/2}  +  \lambda_{{f_{m,k}}} \right)^2 \hspace*{-0,08cm} w(u)^{-1} }
{|{\check{\phi}}_n(u)|^2| |\phi(u)|^2 n}
    \right\},
\end{equation*}
with $\lambda_{{f_{m,k}}}$ defined as in Lemma \ref{Lemma_Abweichung_Erweiterung}. 
We can estimate
\begin{eqnarray}
\notag &&\hspace*{-1,7cm}\phantom{\leq} \E\left[{\sup_{k>m,k\in \N} } \left\{ \left|\left( \tilde{\theta}_k -\tilde{\theta}_m\right) -
\left(\theta_{k} -\theta_{m}  \right)   \right|^2 -\frac{1}{8}\tilde{\strafH}^2(m,k) \right\}_{+}  \right]  \\
\notag &&\hspace*{-1,7cm}\leq  \E\left[\sup_{k>m,k\in \N}  \left\{ \left|\left( \tilde{\theta}_k -\tilde{\theta}_m\right) -
\left(\theta_{k} -\theta_{m}  \right)   \right|^2 -\frac{1}{8}\tilde{\strafH}^2(m,k) \right\}_{+} 1\Big( C(m,k)
\Big)  \right] \\
\notag &&\hspace{-1,7cm}  + \E\left[\sup_{k>m,k\in \N} \left\{ \left|\left( \tilde{\theta}_k -\tilde{\theta}_m\right) -
\left(\theta_{k} -\theta_{m}  \right)   \right|^2 -\frac{1}{8}\tilde{\strafH}^2(m,k) \right\}_{+} 1\Big( C(m,k)^c
\Big)  \right]\hspace*{-0,15cm}.
\end{eqnarray}
The Cauchy Schwarz inequality and the fact that $|\phi(u)|\leq 1$  imply
\begin{align}
\notag  & \phantom{\leq}   \left|\left( \tilde{\theta}_k -\tilde{\theta}_m\right) -
\left(\theta_{k} -\theta_{m}  \right)   \right|^2  \\
\notag &= \left| \frac{1}{2\pi} \int
\Fourier f(-u) \phi' (u) \left( \frac{1}{ { \check{\phi}}_n(u) } -\frac{1}{\phi(u)} \right) \Kerndifmk   \d u  \right|^2 \\
\notag &=  \left| \frac{1}{2\pi} \int
\Fourier f(-u)  \Psi'(u)
\left(\phinneuu - \phinu    \right) \phi(u)  \Kerndifmk  \d u  \right|^2  \\
\notag &\leq  \frac{1}{(2\pi)^2} \Bigg\{ \|\Psi'\|_{\infty}^2
\left( \int
\left|\Fourier f(-u)\right|  \left| \phinneuu - \phinu  \right| |\phi(u)| \left|\Kerndifmk  \right|\d u \right)^2 \\
\notag  &\phantom{\leq tt}\wedge \Bigg(\| \Psi'\|_{\lk^2}^2 \d x  \int
 \left|\Fourier f(-u) \right|^2  \left| \phinneuu -\phinu \right|^2\hspace*{-0,08cm} |\phi(u)|^2\hspace*{-0,06cm}
\left| \Kerndifmk \right|^2 \d u \Bigg) \Bigg\}. 
\end{align}
The definition of $C(m,k)$ readily implies that on this set, 
\begin{eqnarray*}
\notag &&\hspace*{-0,5cm}\phantom{\leq} \frac{1}{(2\pi)^2} \Bigg\{
\left( \|\Psi'\|_{\infty} \int
\left|\Fourier f(-u)\right|  \left| \phinneuu - \phinu  \right| |\phi(u)| \left|\Kerndifmk  \right|\d u \right)^2 \\
\notag &&\wedge \Bigg(\|\Psi'\|_{\lk^2}^2   \int
 \left|\Fourier f(-u) \right|^2  \left| \phinneuu -\phinu \right|^2\hspace*{-0,08cm} |\phi(u)|^2\hspace*{-0,06cm}
\left| \Kerndifmk \right|^2 \d u \Bigg) \Bigg\} \\
&&\hspace*{-0,5cm}\leq    \left( \left(\frac{5}{2}  \kappa \right)   (\log n)^{1/2}    +  \lambda_{{f_{m,k}}} \right)^2
n^{-1} \tilde{\sigma}^2_{m,k}.
\end{eqnarray*}
We use the trivial  observation  that we always have \mbox{$x_{{f_{m,k}}}^2 \leq \tilde{\sigma}^2_{m,k}$} and hence
\mbox{$\lambda_{{f_{m,k}}}^2 \leq c_1 \tilde{\lambda}_{m,k}^2$} as well as the fact that, by definition,
\begin{equation*}
\tilde{\strafH}^2(m,k) \geq 8 \left( \frac{5}{2}\kappa \left(\log n\right)^{1/2} +\sqrt{c_1} \tilde{\lambda}_{m,k} \right)^2
n^{-1} \tilde{\sigma}_{m,k}^2 ,
\end{equation*}
to conclude that the last line is smaller
than $\frac{1}{8}\tilde{\strafH}^2(m,k)$. We have thus shown 
\begin{equation*}
 \label{E_sup_veranchlaessigbar_2}\hspace*{-0,3cm}\E \left[{\sup_{k> m,k\in \N}}
\left\{ \left|\left( \tilde{\theta}_k- \tilde{\theta}_m\right) - (\theta_k - \theta_m)  \right|^2
- \frac{1}{8}\tilde{\strafH}^2(m,k) \right\}_+ 1(C(m,k))  \right] = 0.
\end{equation*}
It remains to show that the remainder term is negligible. The definition of $\frac{1}{{\check{\phi}}_n}$ implies
that we always have
$\frac{1}{|{\check{\phi}}_n|^2} \leq \kappa^{-2} (\log n)^{-1} n$.
We can thus  estimate
\begin{eqnarray*}
\notag &\phantom{\leq}&\left|\left( \tilde{\theta}_k -\tilde{\theta}_m\right) -
\left(\theta_{k} -\theta_{m}  \right)   \right|^2  \\
\notag &{\leq} &  \frac{1}{2\pi^2}  \Bigg\{ C_1
 \int
\left| \Fourier f(-u)\right|^2  \left| \phinneuu  -\phinu   \right|^2  |\phi(u)|^2 \left|\Kerndifmk  \right|^2 \d u  \\
\notag &\phantom{\leq} & \phantom{\frac{1}{2\pi^2}  \Bigg\{  } \wedge  C_2\left(\int
\left| \Fourier f(-u)\right|  \left| \phinneuu -\phinu  \right| |\phi(u)|  \left| \Kerndifmk \right|\d u \right)^2 \Bigg\}   \\
&\leq &  \frac{1}{2\pi^2} \left(\kappa^{-1} (\log n)^{-1/2} n^{1/2}+1\right)^2   x_{{f_{m,k}}}^2,
\end{eqnarray*}
with $x_{{f_{m,k}}}$ defined as in Lemma \ref{Lemma_Abweichung_Erweiterung}.
This implies
\begin{eqnarray}
\label{re_re_re}
\notag &\phantom{\leq } &  \E\left[{\sup_{k>m, k\in \N} } \left\{ \left|\left( \tilde{\theta}_k -\tilde{\theta}_m\right) -
\left(\theta_{k} -\theta_{m}  \right)   \right|^2 -\frac{1}{8}\tilde{\strafH}^2(m,k) \right\}_{+} 1\Big( C(m,k)^c
\Big)  \right]  \\
\notag &\leq & {\sum_{k>m,k\in \N}\limits}  \E
\left[   \left\{
\left|\left( \tilde{\theta}_k -\tilde{\theta}_m\right) -
\left(\theta_{k} -\theta_{m}  \right)   \right|^2
-\frac{1}{8}\tilde{\strafH}^2(m,k)
\right\}_{+}
1\Big( C(m,k)^c \Big)
\right]  \\
&\leq & {\sum_{k>m,k\in \N}\limits} \frac{1}{2\pi^2} \left(\kappa^{-1} (\log n)^{-1/2} n^{1/2} +1\right)^2   {x}_{{f_{m,k}}}^2  \PP  \Big( C(m,k)^c\Big).
\end{eqnarray}
Now, Lemma \ref{Lemma_Abweichung_emp_cf_uniform_Erweiterung} and the choice of $\kappa$ imply 
\begin{equation*}
\PP  \Big( C(m,k)^c\Big) =C\kap n^{-2}  x_{{f_{m,k}}}^{-2} (m-k)^{-2} ,
\end{equation*}
so the sum appearing in the last line of formula (\ref{re_re_re}) is readily negligible.

This completes the proof.\end{proof}
\subsubsection{Proof of Theorem \ref{Hauptsatz_adaptiver_Schaetzer_Levy} }In what follows, let $m^*$ denote the oracle cutoff, 
\begin{equation*}
 m^*: =   \argmin_{m\in \mathcal{M}}\limits  \left\{ {\sup_{k\geq m, k \in \mathcal{M}} }|\theta_k - \theta_m|^2
+\pen(m)   \right\}.
\end{equation*}
We start by considering the loss on the set $\{\hat{m}\leq m^*\}$.  We use the estimate
\begin{equation*}
 \label{hatm_klein_1}
 \left|\theta - \hat{\theta}_{\hat{m}} \right|^2 1\left( \left\{\hat{m}\leq m^*\right\}\right) 
 \leq  2\left|\theta - \hat{\theta}_{m^*}\right|^2  1\left( \{\hat{m}\leq m^*\} \right) +
2 \left| \hat{\theta}_{m^*} -\hat{\theta}_{\hat{m}} \right|^2 1\left( \left\{ \hat{m}\leq m^*\right\}\right).
\end{equation*}
Lemma \ref{Neumann_Lemma_uniform} and the definition of the penalty term imply  
\begin{equation*}
\label{hatm_klein_2}
 \E \left[ |\theta  -\thetaneu_{m^*}|^2 \right] \leq 2 |\theta - \theta_{m^*}|^2 +
 2 \E\left[ \left|\theta_{m^*}-\hat{\theta}_{m^*}\right|^2\right] 
\leq  2|\theta - \theta_{m^*}|^2 + 2 C\neukap \pen (m^*).
\end{equation*}
By definition of $\hat{m}$,
\begin{eqnarray}
\label{hatm_klein_3}
\notag &&\hspace*{-1,6cm}\phantom{\leq}   \left| \thetaneu_{m^*} - \thetaneu_{\hat{m}}\right|^2 1 \left(\{\hat{m}\leq m^*\} \right)  \\
\notag &&\hspace*{-1,6cm}\leq  \tilde{\pen}(m^*) +
{\sup_{k> m^*,k\in \mathcal{M}}}  \left\{ \left|\thetaneu_k - \thetaneu_{m^*}\right|^2 \hspace*{-0,1cm}- \tilde{\strafH}^2(m^*,k) \right\}
 +\tilde{\strafH}^2(\hat{m},m^*) 1\left( \{\hat{m}\leq m^*\}\right)\hspace*{-0,05cm}.
\end{eqnarray}
 We can estimate
\begin{eqnarray}
\label{hatm_klein_4}
\notag & \phantom{\leq}  & \sup_{{k>m^*,k\in \mathcal{M}}} \left\{ \left|\hat{{\theta}}_k -\hat{{\theta}}_{m^*}\right|^2
-\widetilde{\strafH}^2(m^*,k)
 \right\}  \\
\notag & \leq  & \sup_{k>m^*,k\in \mathcal{M}}  \left\{ 2 \left|(\hat{{\theta}}_k -\hat{{\theta} }_{m^*})-(\theta_{k} - \theta_{{m^*}}) \right|^2
 -\widetilde{\strafH}^2(m^*,k)  \right\} +    2 \sup_{k>m^*,k\in \mathcal{M}} \left|\theta_{k} - \theta_{ {m^*}}  \right|^2.
\end{eqnarray}
Proposition \ref{wichtigstes_Hilfsresultat}  readily  implies for some positive constant $C$, 
\begin{eqnarray}
\notag \label{hatm_klein_5}
\E \left[ \sup_{{k>m^*,k\in \mathcal{M}}} \left\{ 2 \left|(\hat{{\theta}}_k -\hat{{\theta} }_{m^*})-(\theta_{k} - \theta_{{m^*}}) \right|^2
-\widetilde{\strafH}^2(m^*,k)  \right\}  \right]\leq  C n^{-1} .
\end{eqnarray}
Next, we observe that, by definition of $\tilde{\strafH}^2$ and $\tilde{\pen}$,
\begin{equation*}
 \tilde{\strafH}(\hat{m},m^*) 1\left(\left\{\hat{m}\leq m^*\right\}\right) \leq \tilde{\pen}(m^*).
\end{equation*}
Since we have chosen $\kappa\geq 2\left( \sqrt{4c_1}+\gamma\right)$, we can  apply Corollary \ref{Korollar_Neumann_uniform} to find that for some positive constant
$C\neukap$ depending only on the choice of the constants, 
\begin{equation}
 \label{Strafe_tot}
\E \left[\tilde{\pen}(m^*)   \right] \leq C\neukap  \pen(m^*).
\end{equation}
To do this, we apply the Cauchy-Schwarz inequality to see that 
\begin{equation*}
\E\left[ \tilde{\lambda}_{m^*}^2 \tilde{\sigma}^2\right]  \leq \left(\E\left[ \tilde{\lambda}_{m^*}^4 \right] \right)^{1/2}  \left(\E\left[ \tilde{\sigma}_{m^*}^4 \right]   \right)^{1/2} 
\end{equation*}
and then use  Corollary \ref{Korollar_Neumann_uniform} to derive that 
\begin{equation*}
 \left(\E\left[ \tilde{\lambda}_{m^*}^4 \right] \right)^{1/2}  \left(\E\left[ \tilde{\sigma}_{m^*}^4 \right]   \right)^{1/2}  \leq C\neukap  \lambda_{m^*}^2\sigma_{m^*}^2.
\end{equation*}

Putting the above results together, we have shown that for some positive constant $C\neukap$ depending only on the choice of the constants and some positive constant  $C$ specified in 
in Proposition \ref{wichtigstes_Hilfsresultat}, 
\begin{eqnarray}
\notag  \label{hatm_klein_5}
&\phantom{\leq}& \E \left[ \left|\theta- \hat{\theta}_{\hat{m}}\right|^2 1\left(\left\{\hat{m}\leq m^* \right\} \right)  \right]  \\
\notag &\leq & \max\{10, 8  C\neukap \}   \inf_{m\in \mathcal{M}} \limits \left\{ |\theta - \theta_{{m_n}}|^2 +
{\sup_{k\geq m,k\in \mathcal{M}}\limits} |\theta_k - \theta_m|^2 +\pen(m)    \right\} +  Cn^{-1} .
\end{eqnarray}
This is the desired result for the expected loss on $\{\hat{m}\leq m^*\}$.

It remains to consider the loss on the set $\{\hat{m} > m^*\}$.  We use the estimate
\begin{equation*}
\label{erste_Zerlegung_cutoff_gross}
 |\theta - \hat{\theta}_{\hat{m}} |^2
\leq    3 |\theta- \theta_{m^*}|^2 +3 |\theta_{\hat{m}}-\theta_{m*} |^2  +3 |\theta_{\hat{m}}-\hat{\theta}_{\hat{m}}|^2.
\end{equation*}
First, we clearly have
\begin{equation*}
\label{zweite_Zerlegung_cutoff_gross}
3 \left(  |\theta - {\theta}_{m^*} |^2 +|\theta_{\hat{m}}-\theta_{m^*} |^2       \right)
1\left( \left\{  \hat{m}>m^* \right\} \right) 
\leq   6|\theta - \theta_{m_n}|^2 + 6 {\sup_{k>m*,k\in \mathcal{M}}\limits }|\theta_k - \theta_{m*}|^2.
\end{equation*}
Next, we can estimate
\begin{eqnarray}
\notag &\phantom{=}&   |{\theta}_{\hat{m}}-\hat{\theta}_{\hat{m}} |^2 1 \left( \left\{ \hat{m}>m^*\right\} \right)   =
{\sum_{k>m^*,k\in \mathcal{M}}\limits }  |\theta_k -\hat{\theta}_k |^2  1\left( \left\{\hat{m}=k \right\}\right)\\
\notag &\leq & {\sum_{k<m^*,k\in \mathcal{M}}\limits } \left\{ |\theta_k -\hat{\theta}_k |^2 - \tilde{\pen}(k) \right\}_+
+{\sum_{k>m^*,k\in \mathcal{M}} \limits }  \tilde{\pen}(k) 1 \left( \left\{\hat{m}=k\right\} \right).
\end{eqnarray}
Another application of   Proposition \ref{wichtigstes_Hilfsresultat} gives 
\begin{equation*}
 \underset{k\in \mathcal{M}}{\sum_{k>m^*}\limits}\E \left[ \left\{|\theta_k-\hat{\theta}_k|^2-
\tilde{\pen}(k) \right\}_+   \right]\leq C n^{-1}.
\end{equation*}
Moreover, by definition of $\hat{m}$, we have on $\{\hat{m}=k \}$:
\begin{align}
\notag  & \phantom{\leq}\,\, \tilde{\pen}(k) \,\, \leq \,\,  \tilde{\pen}(m^*)   +\sup_{l>m^*,l\in \mathcal{M}}
\left\{|\hat{\theta}_l -\hat{\theta}_{m^*} |^2  -\tilde{\strafH}^2(m^*,l) \right\}_+   \\
\notag  &\leq\,\,  \tilde{\pen}(m^*)  +  2 \underset{l\in \mathcal{M}}{\sup_{l>m^*}\limits}  \left\{
\left| (\hat{\theta}_l -\hat{\theta}_m^*)- (\theta_l -\theta_{m^*} )\right|^2  -\frac{1}{2}\tilde{\strafH}^2(m^*,l) \right\}_+
+  2\underset{l\in \mathcal{M}}{\sup_{l>m^*} } |\theta_l -\theta_{m^*}|^2.
\end{align}
By Proposition \ref{wichtigstes_Hilfsresultat} and formula (\ref{Strafe_tot}), 
\begin{eqnarray*}
&\phantom{\leq}&  \E \left[\tilde{\pen}(m^*)  +  2{\sup_{l>m^*,l\in \mathcal{M}}\limits}  \left\{
\left| (\hat{\theta}_l -\hat{\theta}_m^*)- (\theta_l -\theta_{m^*} )\right|^2  -\frac{1}{2}\tilde{\strafH}^2(m^*,l) \right\}_+ \right]  \\
&  \leq &  C\neukap  \pen(m^*) + C n^{-1}.
\end{eqnarray*}
Putting the above results together, we have shown
\begin{eqnarray}
\notag &\phantom{\leq} & \E \left[  |\theta - \hat{\theta}_{\hat{m}}|^2 1 \left(\left\{\hat{m}>m^*\right\}\right)      \right]  \\
\notag &\leq&  \max\{11, 3 C\neukap\}  \inf_{m\in \mathcal{M}}\limits \left\{ |\theta- \theta_{{m_n}}|^2 +
{\sup_{k>m^*,k\in \mathcal{M}}\limits} |\theta_k -\theta_{m^*}|^2  +\pen(m) \right\} +  C n^{-1} ,
\end{eqnarray}
which is the desired result for $\{\hat{m}>m^*\}$. This completes the proof.  \hfill $\Box$
\nocite{Cai_Low_2}
\nocite{Orey}
\nocite{Klein_Rio}
\nocite{Tsybakov}
\nocite{Katznelson}
\nocite{Goldenshluger_1}
\nocite{Goldenshluger}
\bibliographystyle{plainnat}
\bibliography{literatur}
\end{document}